\documentclass[runningheads]{llncs}
\usepackage{graphicx}
\usepackage{enumitem}
\usepackage{amsmath}
\usepackage{amssymb}
\usepackage{tikz}
\usetikzlibrary{positioning}

\newcommand{\dom}{\textbf{Dom}}
\newcommand{\tuple}{\vec}
\newcommand {\const} [1][\cdot]{=\!\!(#1)}

\newcommand {\A}{\mathfrak A}
\newcommand {\M}{\mathfrak M}

\newcommand {\All}{\textbf{All}}
\newcommand {\cneg}{\sim \!\!}

\newcommand{\FO}{\mathbf{FO}}
\newcommand{\D}{\mathbf D}
\newcommand{\E}{\mathbf E}
\newcommand{\F}{\mathbf F}

\newcommand{\DD}{\mathcal D}

\begin{document}
\title{Doubly Strongly First Order Dependencies}
\titlerunning{Doubly Strongly First Order Dependencies}
%
\author{Pietro Galliani}
\authorrunning{P. Galliani}
%
\institute{Free University of Bozen-Bolzano\\
Universit\"atsplatz 1 - Piazza Universit\'a, 1
Italy - 39100, Bozen-Bolzano
\email{Pietro.Galliani@unibz.it}}
\maketitle              
\begin{abstract}
Team Semantics is a generalization of Tarskian Semantics that can be used to add to First Order Logic atoms and connectives expressing dependencies between the possible values of variables. Some of these extensions are more expressive than First Order Logic, while others are reducible to it. In this work, I fully characterize the (relativizable) atoms and families of atoms that do not increase the expressive power of First Order Logic with Team Semantics when they and their negations are added to it, separately or jointly (or, equivalently, when they are added to First Order Logic augmented with a contradictory negation connective applicable only to literals and dependency atoms). 
\end{abstract}

\section{Introduction}
Team Semantics generalizes Tarski's semantics for First Order Logic by letting formulas be satisfied or not satisfied by sets of assignments (called \emph{Teams}) rather than by single assignments. This semantics arises naturally from the analysis of the game theoretic semantics of First Order Logic and its extensions: in brief, a team represents a set of possible game states (= variable assignments) that can be reached at some subformula, and a team satisfies a subformula if the existential player has a strategy that is winning for the corresponding subgame for every starting assignment in the team.\footnote{More precisely, the so-called ``lax'' version of Team Semantics -- which is the most common form of Team Semantics, and the only one that we will consider in this work -- is what arises if the existential player is allowed to also play non-deterministic strategies.} This approach works equally well for extensions of First Order Logic whose game-theoretic semantics yield \emph{imperfect information} games, as is the case for \emph{Independence-Friendly Logic} \cite{hintikkasandu89,hintikka96,mann11}, the original motivation for the development of Team Semantics as an equivalent compositional semantics \cite{hodges97}.

Jouko V\"a\"an\"anen \cite{vaananen07} observed that a logic roughly equivalent to Independence-Friendly Logic, but with arguably more convenient formal properties (most importantly \emph{locality}, in the sense that the interpretation of a formula in a team depends only on the restriction of the team to the free variables of the formula), can be obtained by adding to the Team Semantics for First Order Logic, in place of the so-called \emph{slashed quantifiers} (``there exists a $y$, chosen independently from $x$, such that \ldots'') of Independence-Friendly Logic, \emph{functional dependence atoms} $=\!\!(x; y)$ that state that the values of $y$ are \emph{determined} by those of $x$. The resulting logic, called \emph{Dependence Logic}, has been the subject of a considerable amount of research that cannot be summarized here (for an up-to-date introduction, we refer the reader to \cite{sep-logic-dependence}); but, more importantly for the purposes of this work, it was soon noticed that Team Semantics could be used to extend First Order Logic via \emph{other} types of atoms, for example from database theory (see e.g. \cite{gradel13,galliani12}), or via new types of connective (like the ``contradictory negation'' of \cite{vaananen07b}, the ``intuitionistic implication'' of \cite{abramsky09}, or the generalized quantifiers of \cite{engstrom12}) that can have no analogue in Tarskian Semantics as their definition involves possible interactions between different assignments. Other logics based on Team Semantics have disposed with the connectives arising directly from the Game-Theoretic Semantics of First Order Logic, selecting instead different families of connectives: of particular interest in this context is the \textbf{FOT} logic of \cite{kontinen2019logics}, that relates to ordinary First Order Logic not through Game Theoretic Semantics but on the level of \emph{team definability}, in the sense that a family of teams is defined by a \textbf{FOT} formula if and only if the corresponding family of relations is first order definable. 

Team Semantics -- aside from its applications and connections with other areas, which we will not discuss here -- can thus be seen as a generalization of Tarskian Semantics that allows for the construction of new kinds of extensions and fragments of First Order Logic; and while some of these extensions have been studied in some depth by now (see e.g. \cite{kontinennu11,luck2018complexity} for the contradictory negation, \cite{yang10} for the intuitionistic implications, or \cite{galliani12,galliani13b,gallhella13,hannula2015hierarchies,ronnholm2018capturing} for database-theoretic atoms), not much is yet known regarding the \emph{general properties} of the extensions of First Order Logic obtainable through Team Semantics. 

This work is a partial answer to the following question: which extensions of First Order Logic based on Team Semantics are genuinely more expressive than First Order Logic itself, and which ones instead can only specify properties that were already first order definable? This question is the obvious starting point for a classification of Team Semantics-based extensions of First Order Logic; and yet, at the moment only some very partial answers (such as the ones of \cite{galliani2019characterizing,galliani2019nonjumping}) are known. 

The main result of this work is a full characterization -- aside from the technical condition of \emph{relativizability}, that most dependencies of interest satisfy -- of the dependencies that are \emph{doubly strongly first order} in the sense that both they and their negations can be added (jointly or separately) to First Order Logic with Team Semantics without increasing their expressive power; and as we will see, these dependencies are the same ones that are ``safe'' for First Order Logic with Team Semantics augmented with the contradictory negation connective of \cite{vaananen07b} \emph{restricted to literals and dependency atoms}. 
\section{Preliminaries}
\subsection{Notation}
Tuples of variables or constants will be written as $\tuple x = x_1 x_2 \ldots x_n$  and as $\tuple c = c_1 c_2 \ldots c_k$ respectively, and their lengths will be written as $|\tuple x| = n$ and $|\tuple c| = k$. When the context is clear we will freely treat tuples as if they were sets, writing e.g. $c_i \in \tuple c$ for ``$c_i$ occurs in $\tuple c$'', $\tuple c \subseteq \tuple d$ for ``every $c_i$ occurring in $\tuple c$ also occurs somewhere in $\tuple d$'' and $M \backslash \tuple m$ for the set $\{m \in M : m \not \in \tuple m\}$. Given any set $A$, we will write $\mathcal P(A)$ for the powerset of $A$, i.e. $\mathcal P(A) = \{B : B \subseteq A\}$.

First-order models $\mathfrak M$, $\mathfrak A$, $\mathfrak B$, \ldots will have domains of discourse $\dom(\mathfrak M) = M$, $\dom(\mathfrak A) = A$, $\dom(\mathfrak B) = B$ and so on, and we will always assume that such domains have at least two elements. When no ambiguity is possible, we will freely write the relation symbols $R, S, \ldots$ and the constant symbols $a, b, c, \ldots$ for their corresponding interpretations, e.g. we will write $\M = (M, R, a, b)$ instead of $\M = (M, R^\M, a^\M, b^\M)$ for the model $\M$ with domain $M$ whose signature contains a relation symbol $R$ (whose interpretation is the relation also written as $R$) and two constant symbols $a$ and $b$ (whose interpretations in $\M$ we will also write as $a$ and $b$). We will use the $:=$ symbol to represent assignments of relations and elements to relation and constant symbols, writing for example $(M, R:=S, a:=m)$ for the first order model with domain $M$, with the $k$-ary relation symbol $R$ interpreted as the relation $S \subseteq M^k$, and with the constant symbol $a$ interpreted as the element $m \in M$. In this work we will not need to consider first order models with function symbols in their signature, but the results of this work do not depend on their absence. If $\M$ is a model with domain $M$ over the empty signature, we will sometimes write $M \models \phi$ instead of $\M \models \phi$.

If $s: V \rightarrow M$ is a variable assignment from some set $V$ of variables to the domain $M$ of some model $\M$, $v$ is some variable symbol (not necessarily in $V$) and $m \in M$ is an element, we will write $s[m/v]: V \cup \{v\} \rightarrow M$ for the variable assignment obtained from $s$ by setting $s[m/v](v) = m$. Likewise, if $\tuple m = m_1 \ldots m_k$ and $\tuple v = v_1 \ldots v_k$ are tuples of elements of $M$ and of pairwise distinct variables, we write $s[\tuple m/\tuple v]$ for the assignment obtained from $s$ by assigning each variable $v_i \in \tuple v$ to the corresponding element $m_i \in \tuple m$. Given a tuple $\tuple v = v_1 \ldots v_k$ of variables and an assignment $s$, we will write $s(\tuple v)$ for the tuple of elements $s(v_1) \ldots s(v_k)$.

We will write $\FO$ for First Order Logic, and we will assume that all expressions of First Order Logic are in Negation Normal Form (i.e. negation appears only in front of atomic subformulas). For any first order formula $\phi \in \FO$, we will write $\lnot \phi$ for the Negation Normal Form expression logically equivalent to the negation of $\phi$. Repeated existential or universal quantifiers will be often omitted: for example, if $\tuple x = x_1 x_2 x_3$, we will write $\exists \tuple x P \tuple x$ or $\exists x_1 x_2 x_3 P x_1 x_2 x_3$ in place of $\exists x_1 \exists x_2 \exists x_3 P x_1 x_2 x_3$, and similarly for $\forall$. As is usual, $\top$ will represent the always-true literal and $\bot$ will represent the always-false one.

Given two logics $\textbf{L}_1$ and $\textbf{L}_2$ (generally extensions $\FO(\mathcal C, \DD)$ of First Order Logic via some connectives $\mathcal C$ and some atoms $\DD$), we will write $\textbf{L}_1 \equiv \textbf{L}_2$ if every sentence (\emph{not} formula) of $\textbf{L}_1$ is equivalent to some sentence of $\textbf{L}_2$, and vice versa; and we will write $\textbf{L}_1 \succ \textbf{L}_2$ if every sentence of $\textbf{L}_2$ is equivalent to some sentence of $\textbf{L}_1$, but there are sentences of $\textbf{L}_1$ that are not equivalent to any sentence of $\textbf{L}_2$.
\subsection{Team Semantics}
In this section, we will briefly recall the main definitions and some useful properties of Team Semantics.
\begin{definition}[Team]
	Let $\mathfrak M$ be a first order model with domain of discourse $\dom(\mathfrak M) = M$ and let $V$ be a finite set of variables. Then a team $X$ over $\M$ with domain $\dom(X) = V$ is a set of variable assignments $s: V \rightarrow M$. Given such a team and some tuple of variables $\tuple v \subseteq \dom(X)$, we will write $X(\tuple v)$ for the $|\tuple v|$-ary relation $\{s(\tuple v) : s \in X\} \subseteq M^{|v|}$. 
\end{definition}
\begin{definition}[Team Semantics for First Order Logic]
	Let $\M$ be a first order model, let $\phi$ be a first order formula in Negation Normal Form\footnote{It is possible to formulate Team Semantics for expressions not in Negation Normal Form, but that leads to more elaborate definitions -- one has to take care of truth and falsity separately, since in general the Principle of the Excluded Middle fails to hold in Team Semantics-based extensions of $\FO$ for the dual negation operator -- and not much is gained by doing that. In this work, we will only apply Team Semantics to logical expressions in Negation Normal Form.} over the signature of $\M$, and let $X$ be a team over $\M$ whose domain contains the free variables of $\phi$. Then we say that $X$ satisfies $\phi$ in $\M$, and we write $\M \models_X \phi$, if this follows from the following rules:
\begin{description}
    \item[TS-lit:] If $\phi$ is a first order literal, $\M \models_X \phi$ if and only if, for all assignments $s \in X$, $\M \models_s \phi$ in the usual sense of Tarski's Semantics; 
    \item[TS-$\vee$:] $\M \models_X \phi_1 \vee \phi_2$ if and only if $X = Y \cup Z$ for two $Y, Z \subseteq X$ such that $\M \models_Y \phi_1$ and $\M \models_Z \phi_2$;
    \item[TS-$\wedge$:] $\M \models_X \phi_1 \wedge \phi_2$ if and only if $\M \models_X \phi_1$ and $\M \models_X \phi_2$; 
    \item[TS-$\exists$:] $\M \models_X \exists v \psi$ if and only if there exists some function $H: X \rightarrow \mathcal P(M) \backslash \{\emptyset\}$ such that $\M \models_{X[H/v]} \psi$, where $X[H/v] = \{s[m/v] : s \in X, m \in H(s)\}$;
    \item[TS-$\forall$:] $\M \models_X \forall v \psi$ if and only if $\M \models_{X[M/v]} \psi$, where $X[M/v] = \{s[m/v] : s \in X, m \in M\}$.
\end{description}
    A sentence $\phi$ is said to be true in a model $\M$ if and only if $\M \models_{\{\emptyset\}} \phi$, where $\{\emptyset\}$ is the team containing only the empty assignment.
\end{definition}
A couple of commonly known and often useful derived rules for Team Semantics are the following:
\begin{proposition}
	If $\tuple v$ is a tuple of pairwise distinct variables,
	\begin{description}
		\item[TS-$\exists$-vec:] $\M \models_X \exists \tuple v \psi$ if and only if there exists some $H: X \rightarrow \mathcal P(M^{|\tuple v|}) \backslash \{\emptyset\}$ such that $\M \models_{X[H/\tuple v]} \psi$, where $X[H/\tuple v] = \{s[\tuple m/\tuple v] : \tuple m \in H(\tuple s)\}$; 
		\item[TS-$\forall$-vec:] $\M \models_X \forall \tuple v \psi$ if and only if $\M \models_{X[M/\tuple v]} \psi$, where $X[M/\tuple v] = \{s[\tuple m/\tuple v] : \tuple m \in M^{|\tuple v|}\}$. 
	\end{description}
	\label{propo:quant-vec}
\end{proposition}
These two rules can be proved from TS-$\exists$ and TS-$\forall$ respectively via straightforward induction over the length of $\tuple v$, and are also true for all Team Semantics-based extensions of First Order Logic discussed in this work. 

For first order logic proper, Team Semantics reduces to ordinary Tarskian Semantics:
\begin{proposition}(\cite{vaananen07}, Corollary 3.32)
Let $\M$ be a first order model, let $\phi$ be a first order formula in Negation Normal Form over the signature of $\M$, and let $X$ be a team over $\M$ whose domain contains the free variables of $\phi$. Then $\M \models_X \phi$ if and only if, for all $s \in X$, $\M \models_s \phi$ in the sense of Tarskian Semantics. 

In particular, if $\phi$ is a sentence, $\M \models \phi$ in the sense of Team Semantics if and only if $\M \models \phi$ in the sense of Tarskian Semantics. 
	\label{propo:flat}
\end{proposition}

What is then the purpose of Team Semantics? Historically, Team Semantics was developed in \cite{hodges97} as a means to provide a compositional semantics to Independence-Friendly Logic, which generalizes First Order Logic $\FO$ by turning its Game-Theoretic Semantics into a game of imperfect information \cite{hintikkasandu89,hintikka96,mann11}. The same logic can be defined by Team Semantics by introducing rules for ``slashed'' quantifiers $(\exists v / \tuple w) \psi$ requiring that the choice of $v$ for any assignment $s$ is \emph{independent} from the values of the variables in $\tuple w$, in the sense that for any two assignments that only disagree with respect to $\tuple w$ the same value for $v$ is selected; but, as observed by V\"a\"an\"anen in \cite{vaananen07}, an essentially expressively equivalent logic may be obtained by introducing instead a (functional) \emph{dependence atom} $=\!\!(\tuple x; \tuple y)$, where $\tuple x$ and $\tuple y$ are tuples of variables, whose intended meaning is ``The value of $\tuple y$ is determined by the value of $\tuple x$'' and whose corresponding rule in Team Semantics is 
\begin{description}
    \item[TS-dep:] $\M \models_X =\!\!(\tuple x; \tuple y)$ if and only, if for any two $s, s' \in X$ 
    \[
        s(\tuple x) = s'(\tuple x) \Rightarrow s(\tuple y) = s'(\tuple y)
    \]
\end{description}
This rule corresponds precisely to the database-theoretic notion of \emph{functional dependence} \cite{armstrong74}; and it was soon recognized that other dependency notions may also be studied in the same context, such as \emph{independence statements} \cite{gradel13}
\begin{description}
\item[TS-ind:] $\M \models_X \tuple x \bot_{\tuple y} \tuple z$ if and only if for any two $s, s' \in X$ with $s(\tuple y) = s'(\tuple y)$ there exists some $s'' \in X$ with $s''(\tuple x \tuple y) = s(\tuple x \tuple y)$ and $s''(\tuple y \tuple z) = s'(\tuple y \tuple z)$
\end{description}
which as per \cite{engstrom12} have a close connection with database-theoretic \emph{embedded multivalued dependencies}, \emph{inclusion dependencies} \cite{galliani12,gallhella13,hannula2015hierarchies}
\begin{description}
\item[TS-inc:] $\M \models_X \tuple x \subseteq \tuple y$ if and only if, for all $s \in X$, there exists some $s' \in X$ with $s(\tuple x) = s'(\tuple y)$.
\end{description}
and \emph{exclusion dependencies} \cite{galliani12,ronnholm2019expressive}
\begin{description}
\item[TS-exc:] $\M \models_X \tuple x | \tuple y$ if and only if, for all $s, s' \in X$, $s(\tuple x) \not = s'(\tuple y)$.
\end{description}

But what is, in general, an atom in Team Semantics? As observed in \cite{kuusisto2015double}, such an atom can be defined in the same manner in which \emph{generalized quantifiers} are defined in First Order Logic: 
\begin{definition}[Generalized Dependency]
	Let $k \in \mathbb N$. Then a $k$-ary \emph{generalized dependency atom} is a class $\D$, closed under isomorphisms, of models $\M = (M, R)$ over the signature $\{R\}$, where $R$ is a $k$-ary relation. Such an atom gives rise to the rule
\begin{description}
\item[TS-$\D$:] $\M \models_X \D \tuple x$ if and only if $(\dom(\M), R:=X(\tuple x)) \in \D$
\end{description}
where $(\dom(\M), R:=X(\tuple x))$ is the first order model with domain $\dom(\M)$ and signature $\{R\}$, where the relation symbol $R$ is interpreted as $X(\tuple x) = \{s(\tuple x) : s \in X\}$.
\end{definition}
\begin{remark}
	\label{rem:noterms}
	Sometimes dependency atoms are permitted to apply not only to tuples of variables but more in general to tuples \emph{terms}, thus allowing expressions like $=\!\!(f(x); f(f(x)))$ where $f$ is a function symbol. For simplicity, in this work we ignore this possibility and assume that dependencies are always applied to variables. All the results in this work, nonetheless, transfer seamlessly to the case in which dependencies involving terms are allowed: indeed, one can write e.g. $\exists z w (z = f(x) \wedge w = f(f(x)) \wedge =\!\!(z; w))$ in place of $=\!\!(f(x); f(f(x)))$ and preserve the same meaning.
\end{remark}
\begin{definition}[$\FO(\mathcal D)$]
Let $\mathcal D = \{\D_1, \D_2, \ldots\}$ be a set of generalized dependencies. Then we write $\FO(\mathcal D)$ to the logic obtained by taking First Order Logic (with Team Semantics) and adding to it all the generalized dependency atoms $\D \in \mathcal D$. 
\end{definition}
\begin{remark}
We will write $\FO(\D)$ rather than $\FO(\{\D\})$, $\FO(\D_1, \D_2, \ldots)$ rather than $\FO(\{\D_1, \D_2, \ldots\})$, and $\FO(\mathcal D, \mathcal E)$ rather than $\FO(\mathcal D \cup \mathcal E)$.
\end{remark}

An important class of generalized dependencies is that of \emph{first order} generalized dependencies: 
\begin{definition}[First Order Generalized Dependencies]
A $k$-ary generalized dependency $\D$ is \emph{first order} if there exists some first order sentence $\D(R)$, over the signature $\{R\}$ where $R$ is a $k$-ary relation symbol, such that 
\[
    \D = \{(M, R): (M, R) \models \D(R)\}.
\]
\end{definition}
\begin{remark}
As the above definition shows, we will use the same symbol for a first order dependency and for the first order sentence that defines it: for $\D$ first order, $(M, R) \in \D$ if and only if $(M, R) \models \D(R)$.
\end{remark}

If $\D$ is a first order generalized dependency, it is immediate to see that rule TS-$\D$ is equivalent to 
\begin{description}
\item[TS-$\D$-FO:] $\M \models_X \D \tuple x$ if and only if $(\dom(\M), R:=X(\tuple x)) \models \D(R)$. 
\end{description}
Functional dependence atoms, independence atoms, inclusion atoms, and exclusion atoms (like most other dependency atoms studied in the context of Team Semantics so far) are first order dependencies, generated respectively by the first order sentences 
\begin{align}
    & \forall \tuple x \tuple y \tuple y' ((R \tuple x \tuple y \wedge R \tuple x \tuple y') \rightarrow \tuple y = \tuple y');\\
    & \forall \tuple x \tuple y \tuple z \tuple x' \tuple z' ((R \tuple x \tuple y \tuple z \wedge R \tuple x' \tuple y \tuple z') \rightarrow R \tuple x \tuple y \tuple z');\\
    & \forall \tuple x \tuple y (R \tuple x \tuple y \rightarrow \exists \tuple z (R \tuple z \tuple x));\\
    & \forall \tuple x \tuple y \tuple x' \tuple y' ((R \tuple x \tuple y \wedge R \tuple x' \tuple y') \rightarrow \tuple x \not = \tuple y').
\end{align}
Nonetheless, the logics obtained by adding these dependencies to First Order Logic with Team Semantics (called respectively ``Dependence Logic'', ``Independence Logic'' ``Inclusion Logic'' and ´´Exclusion Logic'') are more expressive than first order logic: more precisely, Dependence Logic, Independence Logic and Exclusion Logic are equiexpressive with Existential Second Order Logic $\Sigma_1^1$ (\cite{vaananen07}, Corollary 6.3; \cite{gradel13}, Corollary 15; \cite{galliani12}, Corollary 4.18 respectively), in the sense that every sentence of any of these logics is equivalent to some sentence of $\Sigma_1^1$ and vice versa\footnote{On the other hand, it is not true that any formula with free variables in any of these three logics is equivalent to some formula in either of the other: in this stronger sense, Dependence Logic and Exclusion Logic are equivalent to each other and properly contained in Independence Logic.} while Inclusion Logic is equivalent to the positive fragment of Greatest Fixed Point Logic.

The fact that adding first order dependencies to First Order Logic with Team Semantics may lead to logics stronger than First Order Logic is a consequence of the higher order character of Team Semantics, and in particular to the second-order existential quantifications implicit in the rules TS-$\vee$ and TS-$\exists$ for disjunction and existential quantification.\footnote{Ultimately, the existential second order character of Team Semantics derives from the existential second order character of Game-Theoretic Semantics, that defines truth in terms of the existence of winning strategies -- that is to say, functions from positions to moves or sets of moves -- in certain semantic games.} This raises immediately a natural (and, for now, unsolved in the general case) problem: can we characterize the generalized dependencies (or the families of dependencies, or even more in general the families of connectives and dependencies) that do \emph{not} increase the expressive power of First Order Logic if added to it?
\subsection{Strongly First Order Dependencies}
\label{subsec:sfo}
In this section, we will recall some partial results related to the problem of characterizing the dependencies that are \emph{Strongly First Order} in the sense of the following definition:
\begin{definition}
A generalized dependency $\D$ is \emph{strongly first order} if and only if $\FO(\D) \equiv \FO$, i.e., if and only if every sentence of $\FO(\D)$ is equivalent to some first order sentence.
Likewise, a family of generalized dependencies $\DD$ is strongly first order if and only if $\FO(\mathcal D) \equiv \FO$. 
\end{definition}

\begin{remark}
Clearly, if $\DD$ is a strongly first order family of dependencies then every $\D \in \DD$ is strongly first order. The converse however is not self-evident: in principle, there could exist two dependencies $\D_1$ and $\D_2$ such that $\FO(\D_1) \equiv \FO$, $\FO(\D_2) \equiv \FO$, but $\FO(\D_1, \D_2) \succ \FO$. In general, one must thus be mindful of the distinction between ``a strongly first order family of dependencies'' and ``a family of strongly first order dependencies''. If the conjecture about the characterization of strongly first order dependencies of \cite{galliani2019nonjumping} held, then indeed $\DD$ would be strongly first order if and only if every $\D \in \DD$ is strongly first order; but at the time of writing that conjecture is unproven. In this work, we will show that a similar result holds for a certain strengthening of the notion of strong first orderness.
\end{remark}

A strongly first order generalized dependency must necessarily be first order:  
\begin{proposition}
Let $\D$ be a strongly first order generalized dependency. Then $\D$ is first order. 
\end{proposition}
\begin{proof}
	Consider the $\FO(\D)$ sentence $\forall \tuple z (\lnot R \tuple z \vee (R \tuple z \wedge \D \tuple z))$. Since $\D$ is strongly first order, this sentence corresponds to some first order sentence $\phi(R)$; and this sentence defines the generalized dependency $\D$, in the sense that $(M, R) \in \D$ if and only if $(M, R) \models \phi(R)$ (this can be verified by applying the rules of Team Semantics). Therefore, $\D$ is defined by a first order sentence.
\end{proof}

A simple and important example of strongly first order dependencies is given by the \emph{constancy atoms}\footnote{The notation $\const[\tuple x]$ for ``$\tuple x$ is constant'' derives from Dependence Logic: indeed, $\const[\tuple x]$ is equivalent to $=\!\!(\emptyset; \tuple x)$, where $\emptyset$ is the empty tuple of variables. Thus, the constancy atom $=\!\!(\tuple x)$ can be seen as the degenerate  functional dependency atom that states that the value of $\tuple x$ is determined by the value of the empty tuple of variables.}
\begin{description}
	\item[TS-Const:] $\M \models_X \const[\tuple x]$ if and only if for all $s, s' \in X$ it holds that $s(\tuple x) = s'(\tuple x)$. 
\end{description}

A common use of constancy atoms is to force an existentially quantified variable (or a tuple of existentially quantified variables) to take only one value in a team, as the following commonly known result (that follows straightforwardly from Proposition \ref{propo:quant-vec} and from the rules of Team Semantics) shows: 
\begin{proposition}
	For all families of dependencies $\DD$, formulas $\phi$ in $\FO(\const, \DD)$, tuples of pairwise distinct variables $\tuple v$, models $\M$ with domain $M$ whose signature contains that of $\phi$, and teams $X$ over the free variables of $\exists \tuple v \phi$, 
	\[
		\M \models_X \exists \tuple v (\const[\tuple v] \wedge \phi) \text{ if and only if } \exists \tuple m \in M^{|\tuple v|} \text{ s.t. } \M \models_{X[\tuple m/\tuple v]} \phi
	\]
	where $X[\tuple m/\tuple v] = \{s[\tuple m/\tuple v] : s \in X\}$. 
	\label{propo:exists1}
\end{proposition}

A more general strongly first order family of generalized dependencies is given by first order \emph{upwards closed dependencies}, that are strongly first order even taken together with each other and with the constancy atom: 
\begin{definition}[Upwards Closed Dependencies]
A k-ary dependency $\D$ is \emph{upwards closed} if and only if 
\[
    (M, R) \in \D, R \subseteq R' \subseteq M^k \Rightarrow (M, R') \in \D
\]
	for all first order models $(M, R)$ with signature $\{R\}$, where $R$ is a $k$-ary relation symbol. We write $\mathcal{UC}$ for the set of all first order upwards closed dependencies. 
\end{definition}
\begin{theorem}[\cite{galliani2015upwards}, Theorem 21]
	$\mathcal{UC} \cup \{\const\}$ is strongly first order. 
    \label{thm:upwards_sfo}
\end{theorem}

An upwards closed, first order dependence atom which will be of some use in this work is the \emph{totality atom} $\All (\tuple x)$, which says that $\tuple x$ takes \emph{all possible values} in the current team: 
\begin{description}
	\item[TS-\textbf{All}:] $\M \models_X \textbf{All}(\tuple x)$ if and only if $X(\tuple x) = \dom(\M)^{|\tuple x|}$.
\end{description}

Also of interest are the families $\DD_0$ and $\DD_1$ of $0$-ary and unary first order dependencies. A $0$-ary first order dependency atom $[\psi] = \{M : M \models \psi\}$ is nothing but a family of models over the empty signature characterized by some first order sentence over the empty signature $\psi$, and $\M \models_X [\psi]$ if and only if $M \models \psi$ (that is to say, $0$-ary dependencies do not ``look'' at the team $X$ but only at the domain $\dom(\M) = M$). A unary dependency atom is instead a family of models $(M, P)$ over the signature $\{P\}$, where $P$ is a unary predicate, characterized by some first order sentence over the signature $\{P\}$. As shown in (\cite{galliani2016strongly}, Proposition 9 and Theorems 9 and 10), both of these families are strongly first order.

A notion closely related to strong first orderness is that of \emph{safety}: 
\begin{definition}[Safe Dependencies]
    Let $\D$ be a generalized dependency and let $\mathcal E$ be a family of dependencies. Then $\D$ is \emph{safe} for $\mathcal E$ if $\FO(\D, \mathcal E) \equiv \FO(\mathcal E)$, that is, if every sentence of $\FO(\D, \mathcal E)$ is equivalent to some sentence of $\FO(\mathcal E)$.
    
    Likewise, if $\mathcal D$ and $\mathcal E$ are two families of dependencies, $\mathcal D$ is safe for $\mathcal E$ if $\FO(\mathcal D, \mathcal E) \equiv \FO(\mathcal E)$. 
    
    A dependency $\D$ or a family of dependencies $\mathcal D$ is said to be \emph{safe} if it is safe for all families of dependencies $\mathcal E$. 
\end{definition}
Clearly, a dependency $\D$ or a family of dependencies $\mathcal D$ is strongly first order if and only if it is safe for $\emptyset$: thus, in this sense, safety is a generalization of strong first orderness. 

However, as shown in (\cite{galliani2020safe}, Theorem 53), strongly first order dependencies are not necessarily safe: in particular, if $\subseteq_1$ represents the ``unary'' inclusion atom $x \subseteq y$, in which $x$ and $y$ are single variables rather than tuples of variables\footnote{Note that, as a generalized dependency, $\subseteq_1$ is \emph{binary}, not a unary: and indeed, it is not strongly first order.}, we have that $\FO(\const, \subseteq_1) \succ \FO(\subseteq_1)$. Nonetheless, as shown in \cite{galliani2019characterizing}, strongly first order dependencies are indeed safe for the class of \emph{downwards closed} dependencies: 
\begin{definition}[Downwards Closed Dependencies]
A k-ary dependency $\D$ is \emph{downwards closed} if and only if 
\[
    (M, R) \in \D, R' \subseteq R \Rightarrow (M, R') \in \D
\]
for all first order models $(M, R)$ with signature $\{R\}$, where $R$ is a $k$-ary relation symbol. We write $\mathcal{DC}$ for the set of all first order downwards closed dependencies. 
\end{definition}
\begin{theorem}[\cite{galliani2019characterizing}, Theorem 3.8]
	Let $\mathcal D$ be a strongly first order family of dependencies and let $\mathcal E \subseteq \mathcal {DC}$ be a family of downwards closed dependencies. Then $\mathcal D$ is safe for $\mathcal E$: $\FO(\mathcal D, \mathcal E) \equiv \FO(\mathcal E)$. 
    \label{thm:sfo_safe_dc}
\end{theorem}

As shown in \cite{galliani2020safe}, totality is safe for any collection of dependencies: 
\begin{proposition}[\cite{galliani2020safe}, Theorem 38]
	$\FO(\All, \DD) \equiv \FO(\DD)$ over sentences for any collection of dependencies $\DD$. 
	\label{propo:all_safe}
\end{proposition}

The main result of \cite{galliani2019characterizing} is a complete characterization of the strongly first order dependencies $\D$ that are downwards closed, have the \emph{empty team property} and are \emph{relativizable}:
\begin{definition}[Empty Team  Property]
	A dependency $\D$ has the \emph{empty team property} if $(M, \emptyset) \in \D$ for all domains of discourse $M$. 
\end{definition}
\begin{definition}[Relativization of a dependency]
    Let $P$ be a unary predicate and let $\D$ be a $k$-ary generalized dependency. Then the relativization of $\D$ to $P$ is the $k$-ary atom $\D^{(P)}$, whose semantics - for models $\M$ whose signature contains the predicate $P$ - is given by 
    \begin{description}
        \item[TS-$\D^{(P)}$:] $\M \models_X \D^{(P)} \tuple x$ if and only if $(P^{\mathcal M}, X(\tuple x)) \in \D$.
    \end{description}
    
	A dependency $\D$ is relativizable if every sentence of $\FO(\D^{(P)})$, i.e. of First Order Logic with Team Semantics augmented with the above rule, is equivalent to some sentence of $\FO(\D)$. A family of dependencies $\DD$ is relativizable if $\FO(\DD^{(P)}) \equiv \FO(\DD)$, where $\DD^{(P)} = \{\D^{(P)} : \D \in \DD\}$.
\end{definition}
As discussed in \cite{galliani2019characterizing}, it is possible to find examples of non-relativizable generalized dependencies.\footnote{The existence of non-relativizable dependencies was first observed in (Barbero, personal communication).} However, all the dependencies studied in the context of Team Semantics so far are relativizable, and no strongly first order non-relativizable dependency is known. Most dependencies of interest have the stronger property of even being \emph{domain independent}  in the sense of the following definition (first found in \cite{kontinen2016decidable}): 
\begin{definition}
    A $k$-ary generalized dependency $\D$ is \emph{domain independent} if and only if 
    \[
        (M, R) \in \D, R \subseteq (M')^k \subseteq M^k \Rightarrow (M', R) \in \D
    \]
\end{definition}
In other words, a dependency $\D$ is domain independent if the truth of an atom $\D \tuple v$ in a team $X$ over a model $\M$ does not depend on the existence or non-existence in $\M$ of elements that do not appear in $X(\tuple v)$. Functional dependence atoms, independence atoms, inclusion atoms and exclusion atoms are all domain-independent, and it would be easy to argue that domain-independence should be a requirement for any reasonable generalized dependency. A simple example of a non-domain-independent, strongly first order dependency is given by the totality atom $\textbf{All}$ seen above. However, this atom is still relativizable: indeed, it is straightforward to check that 
\[
	\textbf{All}^{(P)} (\tuple x) \equiv \left(\bigwedge_{x \in \tuple x} Px\right) \wedge \exists \tuple v \left(\left(\bigvee_{v \in \tuple v} \lnot P v  \vee \tuple v = \tuple x \right)  \wedge \All(\tuple v)\right)
\]
(in order to check that $\tuple x$ contains all tuples of elements in $P^k$, we add to it tuples of elements in $M^k \backslash P^k$ and we check if the result contains all tuples in $M^k$). Currently, no example of non-relativizable strongly first order dependency is known. 

As already mentioned, in \cite{galliani2019characterizing} a full characterization of downwards closed, relativizable, strongly first order dependencies with the empty team property was found: 
\begin{theorem}[\cite{galliani2019characterizing}, Theorem 4.5]
Let $\D$ be a downwards closed, relativizable generalized dependency with the empty team property. Then $\D$ is strongly first order if and only if it is definable in $\FO(\const)$. 
	\label{thm:sfo_dc}
\end{theorem}
\begin{corollary}
Let $\DD$ be a family of downwards closed, relativizable generalized dependencies with the empty team property. Then $\mathcal D$ is strongly first order if and only if all dependencies $\D \in \mathcal D$ are strongly first order. 
\end{corollary}
\begin{proof}
	If $\DD$ is strongly first order then every $\D \in \DD$ is strongly first order, since every sentence of $\FO(\D)$ is a sentence of $\FO(\DD)$. 

	Conversely, suppose that every $\D \in \DD$ is downwards closed, relativizable and strongly first order. Then by Theorem \ref{thm:sfo_dc} every such $\D$ is definable in $\FO(\const)$, and so every sentence of $\FO(\DD)$ is equivalent to some sentence of $\FO(\const)$ (and therefore to some first order sentence).
\end{proof}

Two of the results used in \cite{galliani2019characterizing} to prove Theorem \ref{thm:sfo_dc} will be of use: 
\begin{definition}
	Let $\D$ be any generalized dependency. Then $\D_{\max} = \{(M, R) : (M, R) \in \D, \forall R' \supsetneq R ~ (M, R') \not \in \D\}$ is the dependency satisfied by the \emph{maximal} teams that satisfy $\D$.
\end{definition}
\begin{theorem}[\cite{galliani2019characterizing}, Corollary 4.3]
	Let $\D(R)$ be a downwards closed, strongly first order, relativizable dependency and let $(M, R) \in \D$. Then there exists some $R' \supseteq R$ such that $(M, R') \in \D_{\max}$. 
	\label{thm:max_exist}
\end{theorem}
\begin{theorem}[\cite{galliani2019characterizing}, Proof of Theorem 4.5\footnote{This statement can be extracted from the first part of the proof of Theorem 4.5 of \cite{galliani2019characterizing}. Note in particular that the empty team property is not required for this result, although it is required for the rest of the proof of Theorem 4.5.}]
	Let $\D(R)$ be a downwards closed, strongly first order, relativizable dependency. Then there exists some first order sentence 
	\[
		\D^m(R) = \bigvee_{i = 1}^n \exists \tuple y \forall \tuple x(R \tuple x \leftrightarrow \theta_i(\tuple x, \tuple y)), 
	\]
	where each $\theta_i$ is a first order formula over the empty signature, such that if $(M, R) \in \D_{\max}$ then $(M, R) \models \D^m(R)$. 
	\label{thm:max_def}
\end{theorem}

\cite{galliani2019nonjumping} contains a full characterization of strongly first order, relativizable, \emph{non-jumping}\footnote{In brief, a dependency $\D$ is non-jumping if whenever $(M, R) \in \D$ there is some $R' \supseteq R$ for which $(M, R') \in \D_{\max}$, and for this $R'$ there is no $S$ with $R \subseteq S \subseteq R'$ for which $(M, R) \not \in \D$. Many, but not all, of the generalized dependencies studied in the context of Team Semantics are non-jumping.} dependencies; but as this result is not necessary for the present work, we will not discuss it in detail here. 

More relevant for the purposes of this work is the observation that, even though so far we have been talking about generalized dependency atoms being first order or safe, the same concepts can be applied to \emph{connectives} in Team Semantics. Two connectives of particular interest in the study of Team Semantics are the \emph{Global} (or \emph{Boolean}) \emph{Disjunction} $\phi_1 \sqcup \phi_2$ and the \emph{Possibility Operator} $\diamond \phi$, whose rules are given by 
\begin{description}
\item[TS-$\sqcup$:] $\M \models_X \phi_1 \sqcup \phi_2$ if and only if $\M \models_X \phi_1$ or $\M \models_X \phi_2$; 
\item[TS-$\diamond$:] $\M \models_X \diamond \phi$ if and only if there exists some $Y \subseteq X$, $Y \not = \emptyset$, such that $\M \models_Y \phi$. 
\end{description}
As shown in \cite{galliani2019nonjumping}, global disjunction is not safe for arbitrary dependencies, but it is safe for strongly first order dependencies. We will also need this result for the case of relativized dependencies, so I report the adapted proof (which is not significantly different from the one in \cite{galliani2019nonjumping}):
\begin{proposition}[\cite{galliani2019nonjumping}, Proposition 14]
	If $\DD$ is a strongly first order family of dependencies then every sentence of $\FO(\sqcup, \DD)$ (i.e. of First Order Logic with Team Semantics, augmented with global disjunctions and with the -- possibly relativized -- atoms in $\DD$) is equivalent to some sentence of $\FO$. If $\DD$ is also relativizable, every sentence of $\FO(\sqcup, \DD^{(P)})$ is equivalent to some sentence of $\FO$. 
	\label{propo:sqcupsafe_sfo}
\end{proposition}
\begin{proof}
	It suffices to observe that global disjunction commutes with all connectives: $(\phi \sqcup \psi) \wedge \theta \equiv (\phi \wedge \theta) \sqcup (\psi \wedge \theta)$, $(\phi \sqcup \psi) \vee \theta \equiv (\phi \vee \theta) \sqcup (\psi \vee \theta)$, $\phi \wedge (\psi \sqcup \theta) \equiv (\phi \wedge \psi) \sqcup (\phi \wedge \theta)$, $\phi \vee (\psi \sqcup \theta) \equiv (\phi \vee \psi) \sqcup (\phi \vee \theta)$, $\exists v (\phi \sqcup \psi) \equiv (\exists v \phi) \sqcup (\exists v \psi)$, $\forall v (\phi \sqcup \psi) \equiv (\forall v \phi) \sqcup (\forall v \psi)$.\footnote{Of course, it is not the case that $\forall v (\phi \vee \psi) \equiv (\forall v \phi) \vee (\forall v \psi)$. This illustrates the difference between global disjunction $\sqcup$ and classical disjunction $\vee$: in the first case, the \emph{whole} team has to satisfy the left or the right disjunct, while in the second the team can be ``split'' between the disjuncts.} Also, global disjunction is easily seen to be associative. Therefore, every sentence of $\FO(\sqcup, \DD)$ (resp. $\FO(\sqcup, \DD^{(P)})$) is equivalent to a global disjunction of sentences of $\FO(\DD)$ (resp. $\FO(\DD^{(P)})$, each one of which is equivalent to some first order sentence. Finally, we observe that if $\phi$ and $\psi$ are first order sentences $\M \models \phi \sqcup \psi$ if and only if $\M \models \phi \vee \psi$, because the team $\{\emptyset\}$ containing only the empty assignment cannot be subdivided further: therefore, every sentence of $\FO(\DD)$ (resp. $\FO(\DD^{(P)})$) is equivalent to some first order sentence. 
\end{proof}

On the other hand, in \cite{galliani2020safe} it was shown that $\diamond$ is safe for any collection of dependencies, in the sense that 
\begin{proposition}[\cite{galliani2020safe}, Corollary 42]
Let $\DD$ be any family of generalized dependencies, not necessarily strongly first order. Then every sentence of $\FO(\diamond, \DD)$ (i.e. of First Order Logic with Team Semantics, augmented with the possibility operator $\diamond$ and the atoms in $\DD$) is equivalent to some sentence of $\FO(\DD)$. 
	\label{propo:diamond_safe}
\end{proposition}

Another connective studied in the context of Team Semantics is the \emph{contradictory negation} 
\begin{description}
\item[TS-$\sim$:] $\M \models_X \cneg \phi$ if and only if $\M \not \models_X \phi$. 
\end{description}
Differently from global disjunction and from the possibility operator, contradictory negation is highly unsafe. For example, augmenting Dependence Logic with contradictory negation yields \emph{Team Logic} $\FO(\sim, =\!\!(\cdot;\cdot))$ \cite{vaananen07b}, which is expressively equivalent to full Second Order Logic; and as observed in (\cite{galliani2016strongly}, Corollary 2), even $\FO(\sim, \const)$ is already expressively equivalent to full Second Order Logic. On the other hand, $\sim$ is still ``strongly first order'', in the sense that $\FO(\sim) \equiv \FO$: this is mentioned in \cite{galliani2016strongly} as a consequence of (\cite{galliani2016strongly}, Theorem 4), but it may be verified more simply by observing that if $\phi$ is first order then -- as a consequence of Proposition \ref{propo:flat} -- $\cneg \phi$ is logically equivalent to $\diamond (\lnot \phi)$ and then applying Proposition \ref{propo:diamond_safe}.

We end this section with two simple results that will be of some use:

\begin{proposition}
	Let $\DD$ be a strongly first order, relativizable collection of dependencies. Then every sentence of $\FO(\DD, \const)$ is equivalent to some sentence of $\FO$, and every sentence of $\FO(\DD^{(P)}, \const)$ is also equivalent\footnote{Here by $\DD^{(P)}$ we mean the collection of relativized dependencies $\{\D^{(P)}: \D \in \DD\}$.} to some sentence of $\FO$.\footnote{The non-relativized part of this result could also be proved as a consequence of Theorem \ref{thm:sfo_safe_dc}.}
	\label{propo:const_safe_sfo}
\end{proposition}
\begin{proof}[Sketch]
	Let $\phi$ be some formula of $\FO(\DD, \const)$ or $\FO(\DD^{(P)}, \const)$, and let $\tuple v$ list without repetition all the variables that appear in constancy atoms inside $\phi$. Let $\tuple d$ be a tuple of $|\tuple v|$ new, distinct constant symbols that do not appear in $\phi$, and let $\phi'(\tuple d)$ be the sentence of $\FO(\DD)$ (or $\FO(\DD^{(P)})$) obtained by replacing each constancy atom $\const[\tuple v^i]$ (where, by definition, $\tuple v^i \subseteq \tuple v$) appearing in $\phi$ with the first order literal $\tuple v^i = \tuple d^i$, where $\tuple d^i$ lists the constant symbols corresponding to the variables of $\tuple v^i$ in the same order. 
	
	By structural induction on $\phi$ it can be shown that, for all models $\M$ whose signature contains the signature of $\phi$ and for all teams $X$ over $\M$ whose domain contains the free variables of $\phi$, $\M \models_X \phi$ if and only if there exists a tuple $\tuple m$ of elements of $\M$ such that $(\M, \tuple d := \tuple m) \models_{X} \phi'(\tuple d)$, where $(\M, \tuple d := \tuple m)$ is the expansion of $\M$ obtained by interpreting the constants $d \in \tuple d$ with the corresponding elements $m \in \tuple m$.

	Then, in particular, if $\phi$ is a sentence of $\FO(\DD, \const)$ or $\FO(\DD^{(P)}, \const)$ then $\phi'$ will be a sentence of $\FO(\DD)$ or $\FO(\DD^{(P)})$, which by assumption will be equivalent to some first order sentence $\theta(\tuple d) \in \FO$. But then, if $\tuple w$ is a tuple of $|\tuple d|$ new, distinct variables we will have that $\M \models \phi$ if and only if $\M \models \exists \tuple w \theta(\tuple w/\tuple d)$, where $\theta(\tuple w/\tuple d)$ is the first order sentence obtained from $\theta$ by replacing each constant symbol $d \in \tuple d$ with the corresponding $w \in \tuple w$; and therefore $\phi$ is equivalent to a first order sentence too. 
\end{proof}

\begin{corollary}
	Every sentence of $\FO(\DD_0, \sqcup, \const, \All)$ is equivalent to some sentence of $\FO$. 
	\label{coro:to_FO}
\end{corollary}
\begin{proof}
	By Proposition \ref{propo:sqcupsafe_sfo}, global disjunction is safe for any strongly first order collection of dependencies: therefore, if we can prove that $\FO(\DD_0, \const, \All) \equiv \FO$ we are done. By Proposition \ref{propo:all_safe}, totality is safe for any collection of dependencies, so $\FO(\DD_0, \const, \All) \equiv \FO(\DD_0, \const)$. Finally, since $\DD_0$ is a strongly first order collection of dependencies, by Proposition \ref{propo:const_safe_sfo} we have that $\FO(\DD_0, \const) \equiv \FO$. 
\end{proof}

\subsection{Relations Definable over the Empty Signature}
Finally, in this work we will need a couple of simple results regarding the relations that are definable via First Order Logic formulas over the empty signature:
\begin{definition}
	Let $\M$ be a first order model with domain $M$ and let $\theta(\tuple x, \tuple y)$ be a first order formula. Then we say that a $|\tuple x|$-ary relation $R$ over $\M$ is \emph{defined} by $\theta$ if there is a tuple $\tuple a \in M^{|\tuple y|}$ of elements such that $R = \{\tuple m \in M^{|\tuple x|} : \M \models \theta(\tuple m, \tuple a)\}$. 
\end{definition}
\begin{definition}
	A first order formula $\theta(\tuple x, \tuple y)$ is said to \emph{fix the identity type} of $\tuple y$ if, for every two variables $y_i, y_j \in \tuple y$, $\models \forall \tuple x \tuple y (\theta(\tuple x, \tuple y) \rightarrow y_i = y_j)$ or $\models \forall \tuple x \tuple y(\theta(\tuple x, \tuple y) \rightarrow y_i \not = y_j)$.
\end{definition}
\begin{proposition}
	If a relation $R$ over some model $\M$ is defined by a formula $\theta(\tuple x, \tuple y)$ then it is defined by some formula $\theta'(\tuple x, \tuple y)$ that fixes the identity type of $\tuple y$. 
	\label{propo:fixtype}
\end{proposition}
\begin{proof}
	By assumption, there exists some tuple $\tuple a$ of elements of $\M$ such that $R = \{\tuple m \in M^{|\tuple x|}: \M \models \theta(\tuple m, \tuple a)\}$. Then we can just let 
	\[
		\theta'(\tuple x, \tuple y) = \bigwedge_{a_i = a_j} (y_i = y_j) \wedge \bigwedge_{a_i \not = a_j} (y_i \not = y_j) \wedge \theta(\tuple x, \tuple y)
	\]
	where $a_i$ and $a_j$ range over $\tuple a$.
\end{proof}
\begin{proposition}
	If two nonempty relations $R$ and $S$ over some model $\M$ is defined by the same formula $\theta(\tuple x, \tuple y)$ over the empty signature and $\theta$ fixes the identity type of $\tuple y$ then there is a bijection $\mathfrak h: M \rightarrow M$ that maps $R$ into $S$. 
	\label{propo:autom}
\end{proposition}
\begin{proof}
	By assumption, $R = \{\tuple m \in M^{|x|}: \M \models \theta(\tuple m, \tuple a)\}$ and $S = \{\tuple m \in M^{|x|}: \M \models \theta(\tuple m, \tuple b)\}$ for two tuples of elements $\tuple a = a_1 \ldots a_n$, $\tuple b = b_1 \ldots b_n$ with $|\tuple a| = |\tuple b| = |\tuple y| = n$. Since $\theta(\tuple x, \tuple y)$ fixes the identity type of $\tuple y$ and $R, S \not = \emptyset$, $a_i = a_j$ if and only if $b_i = b_j$ for all $i, j \in 1 \ldots n$. This implies that, as sets, $\tuple a$ and $\tuple b$ contain the same number of elements. Therefore, the sets $M \backslash \tuple a$ and $M \backslash \tuple b$ of the elements of $M$ that do not appear in $\tuple a$ and in $\tuple b$ respectively have the same cardinality and there exists some bijection $\mathfrak g: (M \backslash \tuple a) \rightarrow (M \backslash \tuple b)$. 
	
	Now define $\mathfrak h: M \rightarrow M$ as 
	\[
		\mathfrak h(m) = \left\{
			\begin{array}{l l}
				b_i & \mbox{ if } m = a_i \text{ for some } i \in 1 \ldots n;\\
				\mathfrak g(m) & \text{ otherwise.}
			\end{array}
			\right.
	\]
	This is a well-defined function: indeed, if $m = a_i$ and $m= a_j$ for two $i, j \in 1 \ldots n$ then $a_i = a_j$, and thus -- since $\theta(\tuple x, \tuple y)$ fixes the identity type of $\tuple y$ -- $b_i = b_j$ too. It is 1-1: if $\mathfrak h(m) = \mathfrak h(m')$ then it is necessarily the case that $m$ and $m'$ are both in $\tuple a$ or both not in it, since otherwise would be mapped in $\tuple b$ and the other in $M \backslash \tuple b$. If both are in $\tuple a$ then $m= a_i$, $\mathfrak h(m) = b_i$, $m' = a_j$ and $\mathfrak h(m') = b_j$ for some $i, j \in 1 \ldots |y|$, and $b_i = b_j$, and therefore $m = a_i = a_j = m'$; and if neither is in $\tuple a$ then $m = m'$ because $\mathfrak g$ is a bijection. It is surjective: any value $b_i$ occurring in $\tuple b$ is equal to $\mathfrak h(a_i)$ for the corresponding $a_i$, and since $\mathfrak g$ is a bijection any element in $M \backslash \tuple b$ is the result of applying $\mathfrak h$ to some element in $M \backslash \tuple a$.

	Since $\theta$ is a first order formula over the empty signature and every bijection $M \rightarrow M$ is an automorphism over the empty signature, $\M \models \theta(\tuple m, \tuple a)$ if and only if $\M \models \theta(\mathfrak h(\tuple m), \mathfrak h(\tuple a))$. But $\mathfrak h(\tuple a) = \tuple b$, and so
	\begin{align*}
		\mathfrak h[R] &= \{\mathfrak h(\tuple m) : \M \models \theta(\tuple m, \tuple a)\} = \{\mathfrak h(\tuple m) : \M \models \theta(\mathfrak h(\tuple m), \tuple b)\}\\
		& = \{\tuple m' : \M \models \theta(\tuple m', \tuple b)\} = S
	\end{align*}
	where in the third equality we used the fact that $\mathfrak h$ is surjective.
\end{proof}
\section{Limited Contradictions and Doubly Strongly First Order Dependencies}
Among the additional connectives studied in the context of Team Semantics, the contradictory negation $\cneg \phi$ is both one of the earliest considered and one of the most natural to investigate. However, as we mentioned in Section \ref{subsec:sfo}, adding it to Dependence Logic or even to ``Constancy Logic'' $\FO(\const)$ -- the constancy atom being arguably the simplest kind of non-trivial generalized dependency -- results in a logic as expressive as full Second Order Logic.

This extreme unsafety of the contradictory negation in Team Semantics ultimately derives from the interaction between it and disjunctions and existential quantifiers: indeed, as the rules \textbf{TS-}$\vee$ and \textbf{TS-}$\exists$ show, Team Semantics implicitly performs second-order existential quantifications when evaluating these connectives, and therefore the unrestricted combination of them with contradictory negations can lead to unbounded nestings of existential and universal second-order quantifications. 

Therefore, one may want to investigate the properties of fragments of $\FO(\sim)$ (and extensions thereof by means of dependency atoms) in which such nestings are disallowed or limited, for example by allowing the contradictory negation to appear only before literals and dependency atoms. 
\begin{definition}[Limited Contradiction]
	$\FO(\sim_0)$ is the fragment of $\FO(\sim)$ in which the contradictory negation $\sim$ may appear only in front of first order literals. Likewise, if $\DD$ is any family of dependency atoms, $\FO(\sim_0, \DD)$ is the fragment of $\FO(\sim, \DD)$ in which the contradictory negation $\sim$ may only appear in front of first order literals or dependency atoms. 
\end{definition}
\begin{remark}
	One could wonder at this point if our choice, which we discussed in Remark \ref{rem:noterms}, to forbid complex terms inside generalized dependency atoms does not affect the above definition: even if $=\!\!(f(x); f(f(x)))$ can be reduced to $\exists z w (z = f(x) \wedge w = f(f(x)) \wedge =\!\!(z; w))$, translating $\sim =\!\!(f(x); f(f(x)))$ as $\cneg \exists z w (z=f(x) \wedge w = f(f(x)) \wedge =\!\!(z;w))$ would bring us outside $\FO(\sim_0, =\!\!(\cdot; \cdot))$. This is however not a problem, because that expression is also equivalent to $\exists z w (z = f(x) \wedge w=f(f(x)) \wedge \sim =\!\!(z;w))$: in general, contradictory negation does not commute with existential quantifiers, but it does so when the values of the quantified variables are ``forced'', as they are in this case. 
\end{remark}
Since $\FO(\sim) \equiv \FO$, we have that $\FO(\sim_0) \equiv \FO$ as well. Therefore, it is possible to ask which dependencies $\D$ or families of dependencies $\DD$ are safe for $\FO(\sim_0)$, in the sense that $\FO(\sim_0, \D) \equiv \FO(\sim_0) \equiv \FO$ (respectively $\FO(\sim_0, \DD) \equiv \FO(\sim_0) \equiv \FO$). 

As we will now see, this problem may be reduced to the problem of whether certain families of dependencies are strongly first order for $\FO$. Given a dependency atom $\D$, one can define the complementary atom $\cneg \D$: 
\begin{definition}[Complement of a Dependency Atom]
	Let $\D$ be any generalized dependency atom. Then we write $\cneg \D$ for the generalized dependency atom defined as 
	\[
		\cneg \D = \{(M, R) : (M, R) \not \in \D\}.
	\]
	If $\DD$ is a family of dependencies, we write $\cneg \DD$ for the family $\{\cneg \D : \D \in \DD\}$.
\end{definition}
For all models $\M$, teams $X$ over $\M$, and tuples of variables $\tuple v$ in the domain of $X$, $\M \models_X (\cneg \D) \tuple v$ if and only if $\M \not \models \D \tuple v$: thus, $(\cneg \D) \tuple v \equiv \cneg (\D \tuple v)$ and there is no difference between applying the complement $(\cneg \D)$ of the dependency $\D$ to the tuple of variables $\tuple v$ and applying the contradictory negation operator $\sim$ to the dependency atom $\D \tuple v$. 

\begin{theorem}
	Let $\DD$ be any family of generalized dependencies. Then $\FO(\sim_0, \DD) \equiv \FO$ if and only if $\FO(\DD, \cneg \DD) \equiv \FO$. 
	\label{thm:2sfo-sim0}
\end{theorem}
\begin{proof}
	Suppose that $\FO(\DD, \cneg \DD) \equiv \FO$, and let $\phi$ be a sentence of $\FO(\sim_0, \DD)$. By definition, the contradictory negation $\sim$ may appear in $\phi$ only in front of first order literals or of dependency atoms $\D \tuple v$ for $\D \in \DD$. Now, if $\alpha$ is a first order literal, $\cneg \alpha$ is equivalent to $\diamond \lnot \alpha$: therefore, $\phi$ is equivalent to some sentence of $\FO(\diamond, \DD, \cneg \DD)$. But by Proposition \ref{propo:diamond_safe} the possibility operator $\diamond$ is safe for all collections of dependencies, so $\phi$ is equivalent to some sentence of $\FO(\DD, \cneg \DD)$. Now by assumption $\FO(\DD, \cneg \DD) \equiv \FO$, and so $\phi$ is equivalent to some first order sentence.

	On the other hand, since $(\cneg \D) \tuple v \equiv \cneg (\D \tuple v)$ every sentence of $\FO(\DD, \cneg \DD)$ is equivalent to some sentence of $\FO(\sim_0, \DD)$, and so if $\FO(\DD, \cneg \DD)$ is more expressive than $\FO$ so is $\FO(\sim_0, \DD)$. 
\end{proof}

This justifies the interest in the following definition: 
\begin{definition}[Doubly Strongly First Order Dependencies]
	Let $\D$ be a generalized dependency. Then $\D$ is \emph{doubly strongly first order} if and only if $\{\D, \cneg \D\}$ is strongly first order. Likewise, a family $\DD$ of dependencies is doubly strongly first order if and only if $\DD ~\cup \cneg \DD$ is strongly first order.
\end{definition}

What can we say about doubly strongly first order dependencies, aside from the fact that, by Theorem \ref{thm:2sfo-sim0}, $\D$ is doubly strongly first order if and only if it is safe for $\FO(\sim_0)$? 

A first, easy observation is that the constancy atom $\const$ is doubly strongly first order. Indeed, its negation $\sim \const$ is the ``non-constancy'' operator already studied in \cite{galliani2015upwards}, for which we have that $\M \models_X \sim \const[\tuple x] \Leftrightarrow |X(\tuple x)| \geq 2$. This is a first order, upwards closed dependency, and therefore by Theorem \ref{thm:upwards_sfo} $\{\const, \sim \const\} \subseteq \mathcal{UC} \cup \{\const\}$ is strongly first order. Since $\cneg(\cneg \D)  = \D$ for all dependencies $\D$, it follows at once that $\sim \const$ is also doubly strongly first order: more in general, whenever $\DD$ is doubly strongly first order so is $\cneg \DD$.

A similar argument holds for all downwards closed strongly first order dependencies: 
\begin{proposition}
Let $\DD$ be a strongly first order family of downwards closed generalized dependencies. Then $\DD$ is doubly strongly first order. 
\end{proposition}
\begin{proof}
    Since every $\D \in \DD$ is downwards closed, every $\cneg \D \in \cneg \DD$ is upwards closed: indeed, if $(M, R) \in (\cneg \D)$ then $(M, R) \not \in \D$, and since $\D$ is downwards closed $(M, R') \not \in \D$ for all $R' \supseteq R$, and therefore $(M, R') \in (\cneg \D)$ for all these $R'$. 
    
	Therefore, by Theorem \ref{thm:upwards_sfo}, $\cneg \DD$ is a strongly first order family of dependencies. But by Theorem \ref{thm:sfo_safe_dc}, strongly first order families of dependencies are safe for families of downwards closed dependencies: therefore, $\cneg \DD$ is safe for $\DD$ and 
    \[
        \FO(\cneg \DD, \DD) \equiv \FO(\DD) \equiv \FO.  
    \]
\end{proof}
\begin{remark}
A dependency that is definable in terms of doubly strongly first order dependencies is not necessarily doubly strongly first order itself. Indeed, as we saw, both the constancy atom $\const$ and its negation $\sim \const$ are doubly strongly first order. However, using them both we can define the dependency  
\[
	\not = \!\! (\tuple x; \tuple y) := (\const[\tuple x] \wedge \sim \const[\tuple y]) \vee \top
\]
	that is true in a team $X$ if and only if $X$ contains at least two assignments that agree over $\tuple x$ but not over $\tuple y$. This dependency is strongly first order\footnote{If a dependency is definable in terms of a strongly first order family of dependencies then it is necessarily strongly first order. This is easy to verify, because if $\D$ is definable in terms of $\DD$ then every sentence of $\FO(\D)$ is equivalent to some sentence of $\FO(\DD)$. But as the remark shows, doubly strong first-orderness is different from strong first-orderness in this respect.}, but it is not doubly strongly first order: indeed, its contradictory negation is simply the functional dependency atom $=\!\!(\tuple x, \tuple y)$, which is not strongly first order as Dependence Logic $\FO(=\!\!(\cdot; \cdot))$ is more expressive than $\FO$.
\end{remark}

\section{Characterizing Relativizable, Doubly Strongly First Order Dependencies}
This section contains the main result of this work, that is, a complete characterization of relativizable doubly strongly first order dependencies. As we will see, this characterization will imply that a relativizable dependency $\D$ is doubly strongly first order if and only if both $\D$ and $\cneg \D$ are \emph{separately} strongly first order, and that a family $\DD$ of relativizable dependencies is doubly strongly first order if and only if every $\D \in \DD$ is doubly strongly first order. 

In order to do that, we will first need to generalize some of the results of \cite{galliani2019characterizing} to dependencies that are not necessarily downwards closed.
\begin{definition}[Downwards Closure]
Let $\D$ be any generalized dependency. Then the \emph{downwards closure} $\D^{\downarrow}$ of $\D$ is the dependency defined as 
\[
	\D^{\downarrow} = \{(M, R): \exists R' \supseteq R \text{ s.t. } (M, R') \in \D\}.
\]
\end{definition}
\begin{lemma}
	For all dependencies $\D$, $\D^\downarrow$ is downwards closed and $\D \subseteq \D^\downarrow$. In fact, $\D^\downarrow$ is the smallest downwards closed dependency that contains $\D$, in the sense that if $\E$ is a downwards closed dependency and $\D \subseteq \E$ then $\D^\downarrow \subseteq \E$.
	\label{lem:Ddown}
\end{lemma}
\begin{proof}
	If $(M, R) \in \D^\downarrow$, there exists some $R' \supseteq R$ such that $(M, R') \in \D$. Then if $S \subseteq R$ we also have that $S \subseteq R'$, and so $(M, S) \in \D^\downarrow$. Thus, $\D^\downarrow$ is downwards closed. If $(M, R) \in \D$, then trivially for $R' = R$ we have that $(M, R') \in \D$ and $R' \supseteq R$, so $(M, R) \in \D^\downarrow$: therefore, $\D \subseteq \D^\downarrow$. Finally, suppose that $\E$ is another downwards closed dependency such that $\D \subseteq \E$, and let $(M, R) \in \D^\downarrow$: then there exists some $R' \supseteq R$ such that $(M, R') \in \D$, and therefore $(M, R') \in \E$, and therefore -- since $\E$ is downwards closed and $R \subseteq R'$ -- $(M, R) \in \E$. Thus $\D^\downarrow \subseteq \E$. 
\end{proof}
\begin{lemma}
Let $\D$ be any strongly first order dependency. Then $\D^{\downarrow}$ is also strongly first order; and if $\D$ is relativizable then so is $\D^\downarrow$. 
\label{lem:dcexists}
\end{lemma}
\begin{proof}
Over nonempty teams, $\D^\downarrow \tuple x$ is equivalent to the formula
\[
	\chi(\tuple x) = \forall p q \exists \tuple w ( (p \not = q \vee \tuple x = \tuple w) \wedge \D \tuple w) 
\]
	and that $(\D^{\downarrow})^{(P)} \tuple x$ is equivalent to 
\[
	\chi'(\tuple x) = \forall p q \exists \tuple w ( (p \not = q \vee \tuple x = \tuple w) \wedge \D^{(P)} \tuple w).
\]

	Indeed, if $\M \models_X \D^\downarrow \tuple x$ for $X$ nonempty then there exists some $R' \supseteq X(\tuple x)$ such that $(M, R') \in \D$. Then we can show that $\M \models_X \chi$ by choosing all possible values for $p$ and $q$ (remember that we assume that all first order models have at least two elements) and then letting $\tuple w$ take the same value as $\tuple x$ when $p = q$ and all the values in $R'$ when $p \not = q$. Then in the resulting team $Y$ we will have that $Y(\tuple w) = R'$, and so $\M \models_Y \D \tuple w$ as required. Conversely, if $\M \models_X \chi$ then there exists some choice function $H: X \rightarrow \mathcal P(M^k)$ such that, for $Y = X[M/pq][H/\tuple w]$, $\M \models_Y (p \not = q \vee \tuple x = \tuple w) \wedge \D \tuple w$. Now since $\M \models_Y p \not=q \vee \tuple x = \tuple w$ and $p$ and $q$ take all possible values for all possible values of $\tuple x$, we have that $Y(\tuple w) \supseteq Y(\tuple x) = X(\tuple x)$; and since $\M \models_Y \D \tuple w$, $(M, Y(\tuple w)) \in \D$. Therefore there exists some $R' = Y(\tuple w) \supseteq X(\tuple x)$ such that $(M, R') \in \D$, and hence $(M, X(\tuple x)) \in \D^\downarrow$. A completely similar argument applies to $\chi'$ and $(\D^\downarrow)^{(P)}$. 

	Some care, however, is required if the team $X$ is empty: in that case, $X[M/pq] = \emptyset$ as well, and so the technique used above to ``add'' tuples to $X(\tuple x)$ when $p \not = q$ is not viable. However, we have that $(M, \emptyset) \in \D^\downarrow$ if and only if there exists \emph{any} $R$ such that $(M, R) \in \D$, since the empty set is contained in all relations $R$, and it is easy to check that the $\FO(\D)$ sentence $\exists \tuple x \D \tuple x$ will be true in a model $\M$ if and only if $(M, R) \in \D$ for some $R$. Since $\D$ is strongly first order, this sentence will be equivalent to some first order sentence $\chi_0$. Observe that in all models $\M$ in which $\chi_0$ is true, $\D^\downarrow \tuple z$ will be equivalent to $\bot \sqcup \chi(\tuple z)$;\footnote{Since $\bot$ is false for all assignments, by Rule \textbf{TS-lit} $\M \models_X \bot$ if and only if $X = \emptyset$. Some renaming may be necessary if the tuple $\tuple z$ contains the variables $p$ and $q$.} and in all models $\M$ in which $\chi_0$ is false, $\D^\downarrow \tuple z$ will be false in all teams (including the empty team) and so every $\FO(\D^\downarrow)$ sentence in which at least one $\D^\downarrow$ atom appears will be false. 
	
	Then every sentence $\phi$ of $\FO(\D^\downarrow)$ in which $\D^\downarrow$ occurs at least once will be equivalent to the $\FO(\sqcup, \D)$ sentence $\chi_0 \wedge \phi'$, where $\phi' \in \FO(\sqcup, \D)$ is the result of replacing every such occurrence $\D^\downarrow \tuple z$ of $\D^\downarrow$ in $\phi$ with $\bot \sqcup \chi(\tuple z)$. Since $\D$ is strongly first order, this is equivalent to some first order sentence by Proposition \ref{propo:sqcupsafe_sfo}. Every sentence of $\FO(\D^\downarrow)$ in which $\D^\downarrow$ does not occur, on the other hand, is already a first order sentence, so in conclusion every sentence of $\FO(\D^\downarrow)$ is equivalent to some first order sentence and $\D^\downarrow$ is strongly first order. 

	The argument to show that $\D^\downarrow$ is relativizable is analogous. 
\end{proof}

\begin{lemma}
	Let $\D$ be any generalized dependency. Then $(\D^\downarrow)_{\max} = \D_{\max}$.
\label{lem:dcmax}
\end{lemma}
\begin{proof}
	Suppose that $(M, R) \in \D_{\max}$, that is, $(M, R)$ is maximal for $\D$. Then in particular $(M, R) \in \D$, and therefore $(M, R) \in \D^\downarrow$. Now suppose that $(M, R) \not \in (\D^\downarrow)_{\max}$: then there exists some $R' \supsetneq R$ for which $(M, R') \in \D^\downarrow$. But then, by the definition of $\D^\downarrow$, there exists some $R'' \supseteq R' \supsetneq R$ such that $(M, R'') \in \D$. This contradicts the assumption that $(M, R)$ is maximal for $\D$; therefore, $(M, R) \in (\D^\downarrow)_{\max}$.  
    
	Conversely, suppose that $(M, R) \in (\D^\downarrow)_{\max}$. Then in particular $(M, R) \in \D^\downarrow$, and so there exists some $R' \supseteq R$ for which $(M, R') \in \D$. But then $(M, R') \in \D^\downarrow$ too, since $\D \subseteq \D^\downarrow$; and since $(M, R)$ is maximal for $\D^\downarrow$, the only possibility is that $R = R'$ and therefore that $(M, R) \in \D$. Now suppose that $(M, R)$ is not maximal for $\D$. Then there must exist some $R'' \supsetneq R$ for which $(M, R'') \in \D$. But then $(M, R'') \in \D^\downarrow$ too, which contradicts the maximality of $(M, R)$ for $\D^\downarrow$. Therefore, $(M, R)\in \D_{\max}$. 
\end{proof}
We are now able to generalize Theorems \ref{thm:max_exist} and \ref{thm:max_def}:
\begin{theorem}
	Let $\D$ be a strongly first order, relativizable dependency, and suppose that $(M, R) \in \D$. Then $R$ is contained in some $R'$ such that $(M, R') \in \D_{\max}$; and there exists some first order sentence 
	\begin{equation}
		\D^m(R) = \bigvee_{i = 1}^n \exists \tuple y \forall \tuple x(R \tuple x \leftrightarrow \theta_i(\tuple x, \tuple y)),
		\label{eq:maxdef}
	\end{equation}
	where each $\theta_i$ is a first order formula over the empty signature, such that if $(M, R) \in \D_{\max}$ then $(M, R) \models \D^m(R)$. 
	\label{thm:sfo_max_def}
\end{theorem}
\begin{proof}
	By Lemma \ref{lem:dcexists}, $\D^\downarrow$ is strongly first order and relativizable; and if $(M, R) \in \D$, we have at once that $(M, R) \in \D^\downarrow$. So by Theorem \ref{thm:max_exist}, $R$ is contained into some $R'$ such that $(M, R') \in (\D^\downarrow)_{\max}$. By Lemma \ref{lem:dcmax}, $(\D^\downarrow)_{\max} = \D_{\max}$, so $R$ is contained into some $R'$ such that $(M, R) \in \D_{\max}$. 
	
	Additionally, by Theorem \ref{thm:max_def} we know that all the relations that are maximal for $\D^\downarrow$ (and, hence, for $\D$) satisfy some sentence in the form of Equation (\ref{eq:maxdef}).
\end{proof}

The next lemma will show that, given any strongly first order, relativizable $\D$, we can obtain another strongly first order and relativizable dependency by restricting $\D$ to the subsets of certain first order definable relations:
\begin{lemma}
    Let $\D(R)$ be a $k$-ary strongly first order, relativizable dependency and let $\theta(\tuple x, \tuple y)$ be a first order formula over the empty signature, where $\tuple x$ is a tuple of $k$ distinct variables. Then 
    \[
        \D_\theta = \{(M, R) : (M, R) \in \D, (M, R) \models \exists \tuple y \forall \tuple x(R \tuple x \rightarrow \theta(\tuple x, \tuple y))\}
    \]
    is also strongly first order and relativizable. 
    \label{lem:DtoDtheta}
\end{lemma}
\begin{proof}
	Observe that, by Propositions \ref{propo:flat} and \ref{propo:exists1}, $\D_\theta \tuple x$ is logically equivalent to the $\FO(\D, \const)$ formula 
    \[
	    \D \tuple x \wedge \exists \tuple a (\const[\tuple a] \wedge \theta(\tuple x, \tuple a)).
    \]
	Therefore, every $\FO(\D_\theta)$ sentence is equivalent to some $\FO(\D, \const)$ sentence and hence -- by Proposition \ref{propo:const_safe_sfo} -- to some $\FO$ sentence. Therefore, $\D_\theta$ is strongly first order. 
    
    To show that $\D_\theta$ is also relativizable, observe that its relativization to some unary predicate $P$ can be defined in terms of constancy atoms and of the relativization of $\D$ to $P$, since 
    \[
	    \D_\theta^{(P)} \tuple x \equiv \D^{(P)} \tuple x \wedge \exists \tuple a \left (\const[\tuple a] \wedge \bigwedge_{a \in \tuple a} P a \wedge \theta^{(P)} (\tuple x, \tuple a)\right )
	\]
	where $\theta^{(P)}$ is the relativization (in the usual First Order Logic sense) of $\theta$ to the unary predicate $P$. 

	Therefore, every $\FO(\D_\theta^{(P)})$ sentence is equivalent to some $\FO(\D^{P}, \const)$ sentence, and hence  -- again by Proposition \ref{propo:const_safe_sfo} -- to some first order sentence. 
\end{proof}

The next lemma will be the main ingredient of our characterization of doubly strongly first order dependencies. In brief, it will show that whenever both $\D$ and $\cneg \D$ are strongly first order there can be no infinite ``stair'' of relations satisfying alternatively $\D$ and $\cneg \D$ as per Figure \ref{fig:nostair}:
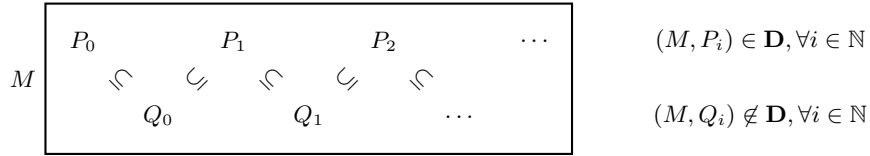
\begin{figure}
		 \begin{center}
		 \begin{tikzpicture}
			 \node[circle,minimum size=2.5em] (P0) at (0,0) {$P_0$};
			 \node[circle,minimum size=2.5em] (Q0) at (1, -1) {$Q_0$}; 
			 \node[circle,minimum size=2.5em] (P1) at (2, 0) {$P_1$}; 
			 \node[circle,minimum size=2.5em] (Q1) at (3,-1) {$Q_1$};
			 \node[circle,minimum size=2.5em] (P2) at (4, 0) {$P_2$}; 
			 \node[circle,minimum size=2.5em] (Q2) at (5, -1) {$\ldots$};
			 \node (P3) at (6, 0) {$\ldots$};

			 \draw[white] (P0) to node[black, sloped]{$\subseteq$} (Q0);
			 \draw[white] (Q0) to node[black, sloped]{$\subseteq$} (P1);
			 \draw[white] (P1) to node[black, sloped]{$\subseteq$} (Q1);
			 \draw[white] (Q1) to node[black, sloped]{$\subseteq$} (P2);
			 \draw[white] (P2) to node[black, sloped]{$\subseteq$} (Q2);
			 
			 \draw[black, thick] (-0.5,-1.5) rectangle (6.5,0.5);
			 \node (M) at (-0.8, -0.5) {$M$};

			 \node (D) at (9, 0) {$(M, P_i) \in \D, \forall i \in \mathbb  N$};
			 \node (notD) at (9, -1) {$(M, Q_i) \not \in \D, \forall i \in \mathbb  N$};
		 \end{tikzpicture}
		 \end{center}
		 \caption{The configuration of $k$-ary relations $P_i$ and $Q_i$ over some domain $M$ that is forbidden by Lemma \ref{lemma:nostair} for $\D$, $\cneg \D$ strongly first order.}
		 \label{fig:nostair}
	 \end{figure}
\begin{lemma}
	Let $\D$ be a $k$-ary dependency, and let $(P_i)_{i \in \mathbb N}$ and $(Q_i)_{i \in \mathbb N}$ be $k$-ary relations over the same domain $M$ such that 
	\begin{enumerate}
		\item For all $i \in \mathbb N$, $(M, P_i) \in \D$; 
		\item For all $i \in \mathbb N$, $(M, Q_i) \in \cneg \D$; 
		\item For all $i \in \mathbb N$, $P_i \subseteq Q_i \subseteq P_{i+1}$. 
	\end{enumerate}
	Then at least one between $\D$ and $\cneg \D$ is not strongly first order.
	\label{lemma:nostair}
\end{lemma}
\begin{proof}
	Let $\D$ and $\cneg \D$ be strongly first order, and suppose that relations $P_i, Q_i$ as per our hypothesis exist over some domain $M$.
	 	 
	 Then $M$ is clearly infinite, and by the L\"owenheim-Skolem Theorem we can assume that it is countable. So we can identify it, up to isomorphism, with $\mathbb N$. Now let the relation $<$ be the usual ordering over the natural numbers and let $P$ and $Q$ be $(k+1)$-ary relations whose interpretations are $\{(\tuple m, i): \tuple m \in P_i\}$ and $\{(\tuple m, i): \tuple m \in Q_i\}$
	respectively. Now consider the model $\M = (\mathbb N, <, P, Q)$. I state that, if $\FO(\D) \equiv \FO(\cneg \D) \equiv \FO$, there is no elementary extension of $\M$ that corresponds to a non-standard model of $\mathbb N$, and in particular that $\M$ has no uncountable elementary extension. This is impossible due to the L\"owenheim-Skolem Theorem, and so one between $\D$ and $\cneg \D$ must not be strongly first order. 
    
    In order to show that $\M$ has no non-standard elementary extensions, consider the $\FO(\cneg \D)$ sentence 
    \begin{equation}
        \exists i (i < d \wedge \forall \tuple w(\lnot P \tuple w i \vee (P \tuple w i \wedge \cneg \D \tuple w)))
        \label{eq:union_negD}
    \end{equation}
	in the signature of $\M$ augmented by some new constant symbol $d$. Since $\cneg \D$ is strongly first order, this sentence is equivalent to some first order sentence $\phi(d)$. I state that $\phi(d)$ is true if and only if there exists a set of indexes $I \subseteq M$ such that $i < d$ for all $i \in I$ and such that $\left(M, \bigcup_{i \in I}P_i)\right) \not \in \D$, where $P_i$ is the relation $\{\tuple m: (\tuple m, i) \in P\}$. Indeed, if such a family of indexes exists, we can satisfy (\ref{eq:union_negD}) by choosing the values of $I$ as the values of the variable $i$, then taking all possible values of $\tuple w$ for all chosen $i$, and then splitting the team by putting in the right disjunct all the assignments $s$ for which $P \tuple w i$ (that is, for which $s(\tuple w) \in P_{s(i)}$); and conversely, if (\ref{eq:union_negD}) can be satisfied, the values that the variable $i$ can take will form a set $I$ of indexes $< d$ such that $\bigcup_{i \in I} P_i$ does not satisfy $\D$. 
    
    Now, for the model $\mathfrak M$ with domain $\mathbb N$ described above no such family $I$ may be found no matter the choice of $d$. Indeed, there will be only finitely many indexes less than $d$, and so if all elements of $I$ are less than $d$ then $\bigcup_{i \in I} P_i = P_{\max(I)}$, which satisfies $\D$. Hence, $\M \models \lnot \exists n \phi(n)$. 
    
    Similarly, the $\FO(\D)$ sentence 
    \begin{equation}
		\exists j (j < d \wedge \forall \tuple w (\lnot Q \tuple w j \vee (Q \tuple w j \wedge \D(\tuple w)))).
		\label{eq:union_D}
	\end{equation}
	is true if and only there exists a set $J$ of indexes $< d$ such that $\bigcup_{j \in J} Q_j$ satisfies $\D$; and as above, this sentence must be equivalent to some first order $\psi(d)$, because $\D$ is strongly first order, and $\M \models \lnot \exists n \psi(n)$ because every $d \in \mathbb N$ has finitely many predecessors.
	
	Now let $\M'$ be any elementary extension of $\M$, and let $d$ be any nonstandard element of it (that is, any element bigger than any $n \in \mathbb N$ in the ordering). Then at least one between $\phi(d)$ and $\psi(d)$ must hold. Indeed, in $\M'$ -- like in $\M$ -- we will have that $P_i \subseteq Q_i \subseteq P_{i+1}$ for all indexes $i \in \mathbb N$; and therefore, 
	\[
	    \bigcup_{i \in \mathbb N} P_i = \bigcup_{i \in \mathbb N}Q_i 
	\]
	where all indexes in $\mathbb N$ are less than our element $d$. If this union satisfies $\D$, $\psi(d)$ holds; and if if instead it does not satisfy $\D$, $\phi(d)$ holds. So $\M' \models (\exists n \phi(n)) \vee (\exists n \psi(n))$ and $\M'$ is not an elementary extension of $\M$, contradicting our premise: therefore, $\M$ cannot have elementary extensions with non-standard elements (and in particular it cannot have uncountable elementary extensions, which is impossible due to the L\"owenheim-Skolem Theorem). 
\end{proof}

The next lemma will provide a ``local'' characterization of the relativizable dependencies $\D$ such that both $\D$ and $\cneg \D$ are strongly first order, by showing that whenever $(M, R) \in \D$ the relation $R$ must satisfy a sentence of a certain form that in turn entails $\D(R)$. Once such a characterization is obtained, deriving a global characterization will be a simple matter of applying the compactness theorem: 
\begin{lemma}
	Let $\D$ be a relativizable dependency such that both $\D$ and $\cneg \D$ are strongly first order and let $\M = (M, R) \in \D$ for $M$ countable. Then there exists a first order sentence $\eta_\M$ of the form
	\begin{equation}
		\eta_\M = \psi \wedge \bigwedge_{i=1}^{n} \exists \tuple y_i (\forall \tuple x (R \tuple x \rightarrow \theta_i(\tuple x, \tuple y_i))) \wedge \bigwedge_{j=1}^{n'} \lnot \exists \tuple z_{j} (\forall \tuple x(R \tuple x \rightarrow \xi_j(\tuple x, \tuple z_j)))
		\label{eq:eta_M}
	\end{equation}
	where $\psi$ is a first order sentence over the empty signature and all the $\theta_i$ and the $\xi_j$ are first order formulas over the empty signature, such that 
	\begin{enumerate}
		\item $\M \models \eta_\M$; 
		\item $\eta_\M \models \D(R)$. 
	\end{enumerate}
	\label{lemma:eta_M}
\end{lemma}
\begin{proof}
	Let $T$ be the theory 
	\begin{align*}
		T =   	& \{\psi: M \models \psi\} ~ \cup\\
			&  \{\exists \tuple y \forall \tuple x (R \tuple x \rightarrow \theta(\tuple x, \tuple y)) :  (M, R) \models \exists \tuple y \forall \tuple x (R \tuple x \rightarrow \theta(\tuple x, \tuple y))\} ~\cup\\
			& \{\lnot \exists \tuple z \forall \tuple x (R \tuple x \rightarrow \xi(\tuple x, \tuple z)) :  (M, R) \not \models  \exists \tuple z \forall \tuple x (R \tuple x \rightarrow \xi(\tuple x, \tuple z))\} 
	\end{align*}
	where $\tuple x$ is a tuple of distinct variables such that $|\tuple x|$ is the arity of $\D$, $\psi$ ranges over all first order sentences over the empty signature, $\tuple y$ and $\tuple z$ range over tuples of distinct variables disjoint from $\tuple x$ of all finite lengths (including the empty tuple of variables, in which case their existential quantification is vacuous), and $\theta(\tuple x, \tuple y)$ and $\xi(\tuple x, \tuple z)$ range over first order formulas with free variables in $\tuple x \tuple y$ (respectively $\tuple x \tuple z$) over the empty signature. If we can show that $T \models \D(R)$, the conclusion follows: indeed, by compactness we can then find a finite theory $T_f \subseteq T$ such that $T_f \models \D(R)$, and then $\bigwedge T_f$ is the required expression in the form of (\ref{eq:eta_M}).

	Suppose that this is not the case, that is, there exists some model $\A$ such that $\A \models T$ but $\A \not \in \D$. By the L\"owenheim-Skolem Theorem we can assume that $\dom(\A)$ is also countable; and since $\A$ and $\M$ satisfy the same sentences over the empty signature, $\A$ and $\M$ have the same cardinality (finite or countably infinite). Thus, up to isomorphism, we can assume that the domain of $\A$ is the same as that of $\M$, i.e., $\A = (M, S)$ for some relation $S$. Also, $R$ is necessarily nonempty: indeed, if $R$ were empty then $\forall \tuple x (R \tuple x \rightarrow \bot)$ would be in $T$, and so since $(M, S) \models T$ the relation $S$ would also be empty, which would contradict the assumption that $(M, R) \in \D$ but $(M, S) \not \in \D$. Then $S$ is also necessarily nonempty, because $\lnot \forall \tuple x (R \tuple x \rightarrow \bot)$ is in $T$.
	
	I aim to prove, by induction on $n$, that for every integer $n \in \mathbb N$ there exists in $M$ a descending chain of relations $S^n_0 \supseteq R^n_0 \supseteq S^n_1 \supseteq R^n_1 \ldots S^n_n \supseteq R^n_n \supseteq R$ such that 
	\begin{enumerate}
		\item $(M, S^n_i) \in \cneg \D$ and $(M, R^n_i)  \in \D$ for all $i = 1 \ldots n$;
		\item Every $R^n_i$, for $0 \leq i \leq n$, is defined by some formula $\theta_i(\tuple x, \tuple y)$ over the empty signature that fixes the identity type of $\tuple y$ and by some tuple of elements $\tuple a_i^n$, in the sense that 
			\[
				R^n_i = \{\tuple m:  M \models \theta_i(\tuple m, \tuple a_i^n)\};
			\]
		\item Every $S^n_i$, for $0 \leq i \leq n$, is defined by some formula $\xi_i(\tuple x, \tuple z)$ over the empty signature that fixes the identity type of $\tuple z$ and by some tuple of elements $\tuple b^n_i$, in the sense that 
			\[
				S^n_i =  \{\tuple m : M \models \xi_i(\tuple m, \tuple b_i^n)\}.
			\]
	\end{enumerate}
	\begin{description}
		\item[Base Case:] 
			Figure \ref{fig:basecase} illustrates the argument for the base case.
		
			Since $(M, S) \not \in \D$, $(M, S) \in \cneg \D$; and therefore, by Theorem \ref{thm:sfo_max_def}, there exists some $S' \supseteq S$ such that $(M, S') \in (\cneg \D)_{\max}$. Moreover, again by Theorem \ref{thm:sfo_max_def}, $S'$ is first order definable over the empty signature: $(M, S') \models \exists \tuple z  \forall \tuple x(S' \tuple x \leftrightarrow \xi_0(\tuple x, \tuple z))$, where $\xi_0(\tuple x, \tuple z)$ is a first order formula over the empty signature which, by Proposition \ref{propo:fixtype}, we can assume fixes the identity type of $\tuple z$. Then, since $S \subseteq S'$, $(M, S) \models \exists \tuple z \forall \tuple x(S \tuple x \rightarrow \xi_0(\tuple x, \tuple z))$; and since $(M, S)$ is a model of the theory $T$, this implies that $(M, R) \models \exists \tuple z \forall \tuple x(R \tuple x \rightarrow \xi_0(\tuple x, \tuple z))$ too.
			
			Now consider the dependency 
			\[
				\E_0 = \{(M', R') \in \D: (M', R') \models \exists \tuple z \forall \tuple x (R' \tuple x \rightarrow \xi_0(\tuple x, \tuple z))\}.
			\]

			By Lemma \ref{lem:DtoDtheta}, $\E_0$ is also strongly first order and relativizable, and $(M, R) \in \E_0$; therefore, by Theorem \ref{thm:sfo_max_def}, there exists some $R_0^0 \supseteq R$ such that $(M, R_0^0) \in (\E_0)_{\max}$, and this $R_0^0$ is definable by some $\theta_0(\tuple x, \tuple y)$ over the empty signature that (again by Proposition \ref{propo:fixtype}) we can assume fixes the identity type of $\tuple y$ and by some tuple $\tuple a_0^0 \in M^{|y|}$, in the sense that $R_0^0 = \{\tuple m : M \models \theta_0^0(\tuple m, \tuple a_0^0)\}$.

			Since $(M, R^0_0) \in (\E_0)_{\max}$, $(M, R^0_0) \in \E_0$. Therefore, $(M, R_0^0) \in \D$, and there exists some tuple $\tuple b_0^0$ in $M$ such that $R_0^0 \subseteq S_0^0 = \{\tuple m : M \models \xi_0(\tuple m, \tuple b_0^0)\}$. Now $S^0_0$ is nonempty, as it contains $R^0_0$ and hence $R$, and $S'$ is nonempty, as it contains $S$, and they are both defined by the same formula $\xi_0(\tuple x, \tuple z)$ that fixes the identity type of $\tuple z$; and therefore, by Proposition \ref{propo:autom}, there exists a bijection $\mathfrak h: M \rightarrow M$ that maps $S'$ into $S_0^0$. This implies that $(M, S_0^0)$ is isomorphic to $(M, S')$, and thus that $(M, S_0^0) \in \cneg \D$ as required. 
				\begin{figure}
		 \begin{center}
		 \begin{tikzpicture}
			 \draw[black, thick] (-0.5,-2.5) rectangle (2.5,0.5);
			 \node (M) at (-0.8, -1) {$M$};

			 \node[circle,draw,black] (R) at (0,-0) {$R$}; 
			 \node[circle,draw,black] (S) at (0,-2) {$S$}; 

			 \draw[black, dashed] (-0.4,-2.4) rectangle (2.4, -1.6);
			 \node at (2.2,-2) (Sp) {$S'$};

			 \draw[black] (-0.4,-0.4) rectangle (1,0.4);
			 \node at (0.8,0) (R00) {$R_0^0$};
			 
			 \draw[black] (1,0.4) rectangle (2.4,-0.4);
			 \node at (2.2,0) (S00) {$S_0^0$};

			 \draw[black, thick, ->] (2.2,-1.5) to node[anchor=east]{$\mathfrak h$} (2.2,-0.5);
		 \end{tikzpicture}
		 \end{center}
				\caption{The base case: building $S_0^0$ and $R_0^0$ given $S$ and $R$. $S$ is contained in a maximal $S'$ that satisfies $\cneg \D$, and this $S'$ is defined by some $\xi_0$ that fixes the identity type of its parameters. If $S$ is contained in some relation defined by $\xi_0$, so is $R$; and therefore there is a maximal $R_0^0$ (defined by some $\theta_0$ that fixes the identity type of its parameters) containing $R$ that satisfies $\D$ and is contained in some relation $S_0^0$ defined by $\xi_0$. Finally, $S'$ and $S_0^0$ must be isomorphic, because they are both nonempty and defined by $\xi_0$, so $S_0^0$ satisfies $\cneg \D$.}
		 \label{fig:basecase}
	 \end{figure}
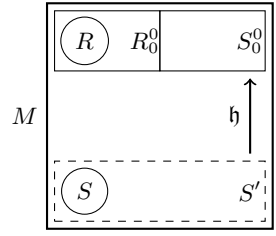

 \item[Induction Case:] Figure \ref{fig:induct_case} illustrates the argument for the induction case. 

	 Suppose that a chain $S^n_0 \supseteq R^n_0 \ldots S^n_n \supseteq R^n_n \supseteq R$ exists as per our induction hypothesis.
			Then, in particular, $R^n_n$ is defined by some $\theta_n(\tuple x, \tuple y)$ that fixes the identity type of $\tuple y$; and since $R \subseteq R^n_n$, $(M, R) \models \exists \tuple y \forall \tuple x(R \tuple x \rightarrow \theta_n(\tuple x, \tuple y))$. But then, since $(M, S) \models T$, $(M, S) \models \exists \tuple y \forall \tuple x(S \tuple x \rightarrow \theta_n(\tuple x, \tuple y))$ too. Now consider the dependency 
			\[
				\F_n = \{(M', R') \in \cneg \D : (M', R') \models \exists \tuple y \forall \tuple x(R' \tuple x \rightarrow \theta_n(\tuple x, \tuple y))\}.
			\]
			By Lemma \ref{lem:DtoDtheta}, $\F_n$ is strongly first order and relativizable, and $(M, S) \in \F_n$; therefore, there exists some $S' \supseteq S$ such that $(M, S') \in (\F_n)_{\max}$, and this $S'$ is defined by some $\xi_{n+1}(\tuple x, \tuple z)$ over the empty signature that fixes the identity type of $\tuple z$. Therefore, since $S \subseteq S'$, $(M, S) \models \exists \tuple z \forall \tuple x(S \tuple x \rightarrow \xi_{n+1}(\tuple x, \tuple z))$; and since $(M, S) \models T$, this implies that $(M, R) \models \exists \tuple z \forall \tuple x(R \tuple x \rightarrow \xi_{n+1}(\tuple x, \tuple z))$ as well. 

			Finally, consider the dependency 
			\[
				\E_{n+1} = \{(M', R') \in \D : (M', R') \models \exists \tuple z \forall \tuple x(R' \tuple x \rightarrow \xi_{n+1}(\tuple x, \tuple z))\}.
			\]
			Again, by Lemma \ref{lem:DtoDtheta}, $\E_{n+1}$ is strongly first order and relativizable, and $(M, R) \in \E_{n+1}$; therefore, there exists some $R^{n+1}_{n+1} \supseteq R$ such that $(M, R^{n+1}_{n+1}) \in (\E_{n+1})_{\max}$. This $R^{n+1}_{n+1}$ will, again, be first order definable over the empty signature by some $\theta_{n+1}(\tuple x, \tuple y)$ that we can assume fixes the identity type of $\tuple y$ by Proposition \ref{propo:fixtype}, and by some tuple $\tuple a^{n+1}_{n+1}$. Furthermore, since $(M, R^{n+1}_{n+1}) \in \E_{n+1}$, it will be the case that $(M, R^{n+1}_{n+1}) \in \D$ and that there exists some $\tuple b^{n+1}_{n+1}$ such that $R^{n+1}_{n+1} \subseteq S^{n+1}_{n+1} = \{\tuple m : M \models \xi_{n+1}(\tuple x, \tuple b^{n+1}_{n+1})\}$. Now since $S^{n+1}_{n+1}$ and $S'$ are defined by the same $\xi_{n+1}(\tuple x, \tuple z)$ and are both nonempty, by Proposition \ref{propo:autom} there exists some isomorphism $\mathfrak g$ between $(M, S')$ and $(M, S^{n+1}_{n+1})$. Therefore, $(M, S^{n+1}_{n+1}) \in \F_n$: thus, $(M, S^{n+1}_{n+1}) \in \cneg \D$, and there exists some $\tuple a^{n+1}_n$ for which $S^{n+1}_{n+1} \subseteq R^{n+1}_n = \{\tuple m : \theta_n(\tuple m, \tuple a^{n+1}_n)\}$. 

			Now, $R^{n+1}_{n}$ and $R^n_n$ are defined by the same formula $\theta_n(\tuple x, \tuple y)$, which fixes the identity type of $\tuple y$, and they are both nonempty since they both contain $R$. Therefore, by Proposition \ref{propo:autom} there exists some bijection $\mathfrak h: M \rightarrow M$ that maps $R^n_n$ into $R^{n+1}_{n}$. Then we can fix, for all $i = 0 \ldots n$, $R^{n+1}_i$ as $\mathfrak h[R^n_i]$ and $S^{n+1}_i$ as $\mathfrak h[S^n_i]$. 

			As required, we will have that $S^{n+1}_0 \supseteq R^{n+1}_0 \supseteq S^{n+1}_1 \supseteq \ldots S^{n+1}_n \supseteq R^{n+1}_n \supseteq S^{n+1}_{n+1} \supseteq R^{n+1}_{n+1} \supseteq R$, because $\mathfrak h$ preserves inclusions. Additionally, $(M, S^{n+1}_i) \in \cneg \D$ and $(M, R^{n+1}_i) \in \D$ for all $i = 0 \ldots n+1$, as required, since $\D$ and $\cneg \D$ are closed under isomorphisms, and for all $i \in 0 \ldots n$ 
			\[
				S^{n+1}_i = \mathfrak h[S^n_i] = \{\mathfrak h(\tuple m) : M \models \xi_i(\tuple m, \tuple b^n_i)\} = \{\tuple m' : M \models \xi_i(\tuple m', \mathfrak h(\tuple b^n_i))\}
			\]
			and 
			\[
				R^{n+1}_i = \mathfrak h[R^n_i] = \{\mathfrak h(\tuple m) : M \models \theta_i(\tuple m, \tuple a^n_i)\} = \{\tuple m' : M \models \theta_i(\tuple m', \mathfrak h(\tuple a^n_i))\}
			\]
			which implies that the $S^{n+1}_i$ and $R^{n+1}_i$ are still defined by the corresponding formulas $\theta_i(\tuple x, \tuple y)$ and $\xi_i(\tuple x, \tuple z)$ (which fix the identity of $\tuple y$ and $\tuple z$ respectively) and by the tuples $\tuple a^{n+1}_i = \mathfrak h(\tuple a^n_i)$, $\tuple b^{n+1}_i = \mathfrak h(\tuple b^n_i)$.

	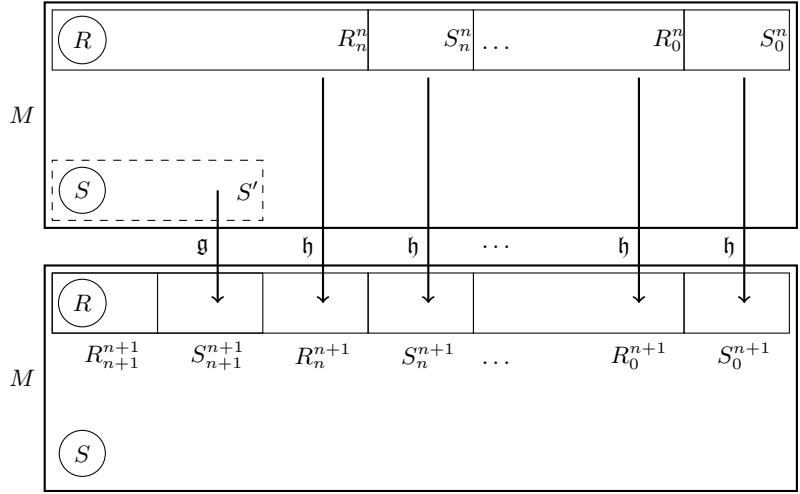
\begin{figure}
		\begin{center}
		\begin{tikzpicture}
			 \draw[black, thick] (-0.5,-2.5) rectangle (9.5,0.5);
			 \node (M) at (-0.8, -1) {$M$};

			 \node[circle,draw,black] (R) at (0,-0) {$R$}; 
			 \node[circle,draw,black] (S) at (0,-2) {$S$}; 

			 \draw[black, dashed] (-0.4,-2.4) rectangle (2.4, -1.6);
			 \node at (2.2,-2) (Sp) {$S'$};

			 \draw[black] (-0.4,-0.4) rectangle (3.8,0.4);
			 \node at (3.6,0) (Rnn) {$R_n^n$};
			 
			 \draw[black] (3.8,0.4) rectangle (5.2,-0.4);
			 \node at (5.0,0) (Snn) {$S_n^n$};
			 \node at (5.5, -0.1) {$\ldots$};

			 \draw[black](5.2,0.4) rectangle (8, -0.4);
			 \node at (7.8, 0) (Rn0){$R_0^n$};
			 
			 \draw[black](8,0.4) rectangle (9.4,-0.4);
			 \node at (9.2,0) (Sn0){$S_0^n$};

			 \draw[black, thick] (-0.5,-6) rectangle (9.5,-3);
			 \node (M2) at (-0.8, -4.5) {$M$};

			 \node[circle,draw,black] (R) at (0,-3.5) {$R$}; 
			 \node[circle,draw,black] (S) at (0,-5.5) {$S$}; 

			 \draw[black] (-0.4,-3.9) rectangle (1,-3.1);
			 \node at (0.4,-4.2) (Rn1n1) {$R_{n+1}^{n+1}$};

			 \draw[black] (1,-3.9) rectangle (2.4,-3.1);
			 \node at (1.8,-4.2) (Sn1n1) {$S_{n+1}^{n+1}$};

			 \draw[black] (-0.4,-3.9) rectangle (3.8,-3.1);
			 \node at (3.2,-4.2) (Rn1n) {$R_n^{n+1}$};
			 
			 \draw[black] (3.8,-3.1) rectangle (5.2,-3.9);
			 \node at (4.6,-4.2) (Sn1n) {$S_n^{n+1}$};
			 \node at (5.5, -4.3) {$\ldots$};

			 \draw[black](5.2,-3.1) rectangle (8, -3.9);
			 \node at (7.4, -4.2) (Rn10){$R_0^{n+1}$};
			 
			 \draw[black](8,-3.1) rectangle (9.4,-3.9);
			 \node at (8.8,-4.2) (Sn10){$S_0^{n+1}$};

			 \draw[black, thick, ->] (1.8,-2) to node[anchor=east]{$\mathfrak g$} (1.8,-3.5);

			 \draw[black, thick, ->] (3.2,-0.5) to (3.2,-3.5);
			 \node at (3.0,-2.75){$\mathfrak h$};

			 \draw[black, thick, ->] (4.6,-0.5) to (4.6,-3.5);
			 \node at (4.4,-2.75){$\mathfrak h$};

			 \node at (5.5, -2.75) {$\ldots$};

			\draw[black, thick, ->] (7.4,-0.5) to (7.4,-3.5);
			\node at (7.2,-2.75){$\mathfrak h$};

			 \draw[black, thick, ->] (8.8,-0.5) to (8.8,-3.5);
			 \node at (8.6,-2.75){$\mathfrak h$};

		 \end{tikzpicture}
		 \end{center}
				\caption{The induction case. $M$ is showed twice to avoid cluttering the figure too much. $R$ is contained in some set defined by $\theta_n$ (i.e. $R^n_n$), therefore so is $S$; and so there exists some maximal $S' \supseteq S$ that satisfies $\cneg \D$ and is contained in some relation defined by $\theta_n$. This $S'$ is defined by some $\xi_{n+1}$. So $R$ is also contained in some relation defined by $\xi_{n+1}$, and there is a maximal $R^{n+1}_{n+1} \supseteq R$  -- defined by some $\theta_{n+1}$ -- that satisfies $\D$, and is contained in some relation $S^{n+1}_{n+1}$ defined by $\xi_{n+1}$. $S^{n+1}_{n+1}$ must be isomorphic to $S'$ via some $\mathfrak g$, because they are both defined by $\xi_{n+1}$, and so it must satisfy $\cneg \D$. Like $S'$, $S^{n+1}_{n+1}$ must be contained in some relation $R^{n+1}_n$ defined by $\theta_n$; and therefore there must be some isomorphism $\mathfrak h$ between $R^n_n$ and $R^{n+1}_n$. So $R^{n+1}_n$ is in $\D$, and we can use $\mathfrak h$ to fix $R^{n+1}_i = \mathfrak h[R^n_i]$ and $S^{n+1}_i = \mathfrak h[S^n_i]$ for all $i=0 \ldots n$.}
		 \label{fig:induct_case}
	 \end{figure}

	\end{description}
	Finally, we can use Lemma \ref{lemma:nostair} to conclude that one among $\D$ and $\cneg \D$ is not strongly first order. Indeed, consider the theory 
	\[
		U = \{\forall \tuple x (P_i \tuple x \rightarrow Q_i \tuple x) \wedge (Q_i \tuple x \rightarrow P_{i+1} \tuple x) : i \in \mathbb N\} \cup \{\D(P_i), \lnot \D(Q_i)\} : i \in \mathbb N\}
	\]
	that states that there is an infinite \emph{ascending} chain $P_0 \subseteq Q_0 \subseteq P_1 \subseteq Q_i \subseteq \ldots$ of relations satisfying alternatively $\D$ and $\lnot \D$. 

	Then $U$ is finitely satisfiable: indeed, for any finite subset $U_f$ of $U$, if $n$ is the highest index for which $P_n$ or $Q_n$ appear in $U_f$, the model with domain $M$ in which $P_0 \ldots P_n$ are interpreted as $R^n_n \ldots R^n_0$ (note the inverse order) and $Q_0 \ldots Q_n$ are interpreted as $S^n_n \ldots S^n_0$ (likewise in inverse order) satisfies $U_f$, since $P_i = R_{n-i} \subseteq S_{n-i} = Q_i$ and $Q_i = S_{n-i} \subseteq R_{n-i-1} = P_{i+1}$.
	
	Therefore, there is a model $\mathfrak N$ that satisfies $U$; and by Lemma \ref{lemma:nostair}, this implies that at least one between $\D$ and $\cneg \D$ is not strongly first order. 
\end{proof}

The next two lemmas will be used in our translation of relativizable, doubly strongly first order dependencies: 
\begin{lemma}
	For all first order models $\M$ with domain $M$, teams $X$ over $M$, and formulas $\theta(\tuple v, \tuple y)$ over the empty signature with $\tuple v$ contained in the variables of $X$, 
	\[
		\M \models_X \exists \tuple y (\const[\tuple y] \wedge \theta (\tuple v, \tuple y)) \Leftrightarrow (M, R:= X(\tuple v)) \models \exists \tuple y \forall \tuple x(R \tuple x \rightarrow \theta(\tuple x, \tuple y)).
	\]
	\label{lemma:exists_theta}
\end{lemma}
\begin{proof}
	Suppose that $\M \models_X \exists \tuple y(\const[\tuple y] \wedge \theta(\tuple v, \tuple y))$. Then by Proposition \ref{propo:exists1} there exists a tuple $\tuple a \in M^{|\tuple y|}$ such that $\M \models_{X[\tuple a/\tuple y]} \theta(\tuple v, \tuple y)$. Thus, by Proposition \ref{propo:flat}, for every assignment $s[\tuple a / \tuple y] \in X[\tuple a/\tuple y]$ we have that $\M \models_{s[\tuple a/\tuple y]} \theta(\tuple v, \tuple y)$ in the ordinary Tarskian sense. Therefore, for every $\tuple m \in X(\tuple v)$ we have that $M \models \theta(\tuple m, \tuple a)$, and therefore $(M, R:=X(\tuple v)) \models \exists y \forall x (R \tuple x \rightarrow \theta(\tuple x, \tuple y))$ as required. 

	Conversely, suppose that $(M, R:=X(\tuple v)) \models \exists \tuple y \forall \tuple x (R \tuple x \rightarrow \theta(\tuple x, \tuple y))$. Then there exists some tuple of elements $\tuple a$ such that, for all $\tuple m \in X(\tuple v)$, $M \models \theta(\tuple m, \tuple a)$. But then $\M \models_{X[\tuple a/\tuple y]} \const[\tuple y]$ and, again by Proposition \ref{propo:flat}, $\M \models_{X[\tuple a/\tuple y]} \theta(\tuple v, \tuple y)$, and so $\M \models_X \exists \tuple y(\const[\tuple y] \wedge \theta(\tuple v, \tuple y))$ as required.
\end{proof}
\begin{lemma}
	For all first order models $\M$ with domain $M$, teams $X$ over $\M$, and formulas $\xi(\tuple v, \tuple z)$ over the empty signature with $\tuple v$ contained in the free variables of $X$,
	\[
		\M \models_X  \top \vee \exists \tuple z (\All(\tuple z) \wedge \lnot \xi(\tuple v, \tuple z)) \Leftrightarrow (M, R:=X(\tuple v)) \models \lnot \exists \tuple z \forall \tuple x (R \tuple x \rightarrow \xi(\tuple x, \tuple z)).
	\]
	\label{lemma:notexists_xi}
\end{lemma}
\begin{proof}
	Suppose that $\M \models_X \top \vee \exists \tuple z (\All(\tuple z) \wedge \lnot \xi(\tuple v, \tuple z))$. Then $X = X_1 \cup X_2$ for two $X_1, X_2 \subseteq X$ and there exists a function $H: X_2 \rightarrow \mathcal P(M^{|\tuple z|})$ such that, for $Y = X_2[H/\tuple z]$, $\M \models_{Y} \All(\tuple z) \wedge \lnot \xi(\tuple v, \tuple z)$. Now take any $\tuple b \in M^{|\tuple z|}$. Since $\M \models_{Y} \All(\tuple z)$, there will be some assignment $s \in Y$ such that $s(\tuple z) = \tuple b$; and since $\M \models_Y \lnot \xi(\tuple v, \tuple z)$, for $\tuple m = s(\tuple v) \in Y(\tuple v) = X_2(\tuple v) \subseteq X(\tuple v)$ we have that $M \not \models \xi(\tuple m, \tuple b)$. Therefore, for every $\tuple b \in M^{|\tuple z|}$ there exists some $\tuple m \in X(\tuple v)$ such that $\lnot \xi(\tuple m, \tuple b)$, and hence $(M, R:=X(\tuple v)) \models \lnot \exists \tuple z \forall \tuple x (R \tuple x \rightarrow \xi(\tuple x, \tuple z))$. 

	Conversely, suppose that $(M, R:=X(\tuple v)) \models \lnot \exists \tuple z \forall \tuple x (R \tuple v \rightarrow \xi(\tuple x, \tuple z))$. Then for every tuple $\tuple b \in M^{|\tuple z|}$ there exists some tuple $\tuple m \in X(\tuple v)$ such that $M \not \models \xi(\tuple m, \tuple b)$, and by the definition of $X(\tuple v)$ we will have that this $\tuple m$ is $s(\tuple v)$ for some assignment in $X$. Then let $X = X_1 \cup X_2$, where $X_2 = \{s \in X :  \exists \tuple b \in M^{|\tuple z|} \text{ s.t. } M \not \models \xi(s(\tuple v), \tuple b)\}$ and $X_1 = X \backslash X_2$. Of course $\M \models_{X_1} \top$; and if we choose the function $H: X_2 \rightarrow \mathcal P(M^{|\tuple z|})$ so that $H(s) = \{\tuple b : M \not \models \xi(s(\tuple v), \tuple b)\}$ for all $s \in X_2$, we have that $H(s) \not = \emptyset$ for all $s \in X_2$, by the definition of $X_2$; that $\M \models_{X[H/\tuple z]} \All(\tuple z)$, since for every possible choice of $\tuple b$ we know that there is some $s \in X$ such that $\lnot \xi(s(\tuple v), \tuple b)$, this $s$ will be assigned to $X_2$ by its definition, and then by the definition of $H$ we will have that $\tuple b \in H(s)$; and that $\M \models_{X_2[H/\tuple z]} \lnot \xi(\tuple v, \tuple z)$, because of Proposition \ref{propo:flat} and of the definition of $H$.
\end{proof}

We now all have all the ingredients to give a complete characterization of doubly strongly first order relativizable dependencies:
\begin{theorem}
	Let $\D(R)$ be a relativizable first order dependency. Then the following are equivalent: 
	
	\begin{enumerate}[label=\roman*)]
		\item $\D$ is equivalent to some sentence of the form
	\begin{equation}
		\bigvee_{k=1}^l \left( \psi_k \wedge
		\bigwedge_{i=1}^{n_k} \exists \tuple y^k_i (\forall \tuple x (R \tuple x \rightarrow \theta^k_i(\tuple x, \tuple y^k_i))) \wedge \bigwedge_{j=1}^{n'_k} \lnot \exists \tuple z^k_{j} (\forall \tuple x(R \tuple x \rightarrow \xi^k_j(\tuple x, \tuple z^k_j)))
		\right)
		\label{eq:DSFO_NF}
	\end{equation}
	where all the $\psi_k$ are first order sentences over the empty vocabulary and all the $\theta_i^k$ and the $\xi_j^k$ are first order formulas over the empty vocabulary; 
	\item Both $\D$ and $\cneg \D$ are definable in $\FO(\DD_0, \sqcup, \const, \All)$; 
	\item $\D$ is doubly strongly first order. 
	\item Both $\D$ and $\cneg \D$ are strongly first order.
	\end{enumerate}
	\label{thm:doublystrongly}
\end{theorem}
\begin{proof}
	Let us show that i) implies ii), that ii) implies iii), that iii) implies iv), and finally that iv) implies i):
	\begin{description}
		\item [i) $\Rightarrow$ ii)]
			Suppose that $\D(R)$ is in the form of Equation (\ref{eq:DSFO_NF}) and, for each first order sentence $\psi$ over the empty signature $\psi$, let $[\psi] \in \DD_0$ be the $0$-ary first order dependency defined as $[\psi] = \{M : M \models \psi\}$. Then $\D \tuple v$ is equivalent to the $\FO(\DD_0, \sqcup, \const, \All)$ formula
	\[
		\bigsqcup_{k=1}^l \left( [\psi_k] \wedge 
		\bigwedge_{i=1}^{n_k} \exists \tuple y^k_i (\const[\tuple y^k_i] \wedge \theta^k_i (\tuple v, \tuple y^k_i)) \wedge \bigwedge_{j=1}^{n'_k} (\top \vee \exists \tuple z^k_j (\All(\tuple z^k_j) \wedge \lnot \xi^k_j(\tuple v, \tuple z^k_j)))
		\right) 
	\]
			where, up to renaming, we can assume that $\tuple v$ is disjoint from all the $\tuple y^k_i$ and $\tuple z^k_j$ and where each $\lnot \xi^k_j$ stands for the corresponding expression in Negation Normal Form. 
			Indeed, by Lemmas \ref{lemma:exists_theta} and \ref{lemma:notexists_xi} -- as well as the rules for conjunction and global disjunction -- the above expression is satisfied by a team $X$ in a model $\M$ if and only if there exists some $k \in 1 \ldots l$ such that 
			\begin{enumerate}
				\item $M \models \psi_k$; 
				\item For all $i \in 1 \ldots n_k$, $(M, X(\tuple v)) \models \exists \tuple y^k_i  \forall \tuple x(R \tuple x \rightarrow \theta^k_i(\tuple x, \tuple y^k_i))$; 
				\item For all $j \in 1 \ldots n'_k$, $(M, X(\tuple v)) \models \lnot \exists \tuple z^k_j \forall \tuple x(R \tuple x \rightarrow \xi^k_j(\tuple x, \tuple z^k_j))$. 
			\end{enumerate}

			These are precisely the conditions for Equation (\ref{eq:DSFO_NF}) to be true in $(M, X(\tuple v))$, that is, for it to be the case that $\M \models_X \D \tuple v$. 

	%
	Likewise, $\cneg \D \tuple v$ is equivalent to 
	\[
		\bigwedge_{k=1}^l \left(
			[\lnot \psi_k] \sqcup 
			\bigsqcup_{i=1}^{n_k} (\top \vee \exists \tuple y^k_i (\All(\tuple y^k_i) \wedge \lnot \theta^k_i(\tuple v, \tuple y^k_i))) \sqcup \bigsqcup_{j=1}^{n'_k} \exists \tuple z^k_j (\const[\tuple z^k_j] \wedge \xi^k_j(\tuple v, \tuple z^k_j))
			\right).
	\]

	Therefore, both $\D$ and $\cneg \D$ are indeed definable in $\FO(\DD_0, \sqcup, \const, \All)$. 
\item [ii) $\Rightarrow$ iii)]
	Since both $\D$ and $\cneg \D$ are definable in $\FO(\DD_0, \sqcup, \const, \All)$, every sentence of $\FO(\D, \cneg \D)$ is equivalent to some sentence of $\FO(\DD_0, \sqcup, \const, \All)$ and therefore -- by Corollary \ref{coro:to_FO} -- to some sentence of $\FO$. Therefore, $\D$ is doubly strongly first order. 
\item  [iii) $\Rightarrow$ iv)] 
	Obvious, since $\FO(\D),\FO(\cneg \D) \preceq \FO(\D, \cneg \D) \equiv \FO$.
	\item [iv) $\Rightarrow$ i)] 
		Suppose that both $\D$ and $\cneg \D$ are strongly first order. Then, by Lemma \ref{lemma:eta_M}, for every countable $\M = (M, R) \in \D$ there exists some first order sentence $\eta_\M$ of the form of Equation (\ref{eq:eta_M}) such that $\M \models \eta_\M$ and that $\eta_\M \models \D(R)$. Now consider the first order theory 
			\[
				T = \{\lnot \eta_\M : \M \in \D, \M \mbox{ is countable}\} \cup \{\D(R)\}
			\]
			Then $T$ is unsatisfiable: indeed, if it had a model then by the L\"owenheim-Skolem Theorem it would have a countable model $\M = (M, R)$, but this is impossible because we would have that $\M \in \D$ (since $\D(R) \in T$) and thus $\M \models \eta_\M$, despite the fact that $\lnot \eta_\M \in T$. Therefore, there is a finite subset of it that is unsatisfiable, or, in other words, we have that 
			\[
				\D(R) \models \bigvee_{k=1}^l \eta_{\M_k}
			\]
			for some finite set $\M_1 \ldots \M_l$ of countable models of $\D$. But each such $\eta_{\M_k}$ entails $\D(R)$, and therefore $\D(R)$ is equivalent to $\bigvee_{k=1}^l \eta_{\M_k}$ which is in the form of Equation (\ref{eq:DSFO_NF}). 
	\end{description}
\end{proof}
As a corollary of the above result, we have that a family of dependencies is doubly strongly first order if and only if every dependency in it is doubly strongly first order: 
\begin{corollary}
Let $\DD$ be a family of relativizable dependencies. Then $\DD$ is doubly strongly first order if and only if every $\D \in \DD$ is doubly strongly first order. 
\end{corollary}
\begin{proof}
	If $\DD$ is doubly strongly first order and $\D \in \DD$, every sentence of $\FO(\D, \cneg \D)$ will also be a sentence of $\FO(\DD, \cneg \DD) \equiv \FO$, and therefore $\D$ is doubly strongly first order. 

	Conversely, suppose that every $\D \in \DD$ is doubly strongly first order and relativizable. Then by Theorem \ref{thm:doublystrongly}, for every $\D \in \DD$ both $\D$ and $\cneg \D$ are definable in $\FO(\DD_0, \sqcup, \const, \All)$. Therefore, every sentence of $\FO(\DD, \cneg \DD)$ is equivalent to some sentence of $\FO(\DD_0, \sqcup, \const, \All)$, and thus -- because of Corollary \ref{coro:to_FO} -- to some first order sentence. Therefore, $\DD$ is doubly strongly first order.
\end{proof}
\section{Conclusions and Further Work}
In this work I provided a complete characterization of the relativizable dependencies that are doubly strongly first order, in the sense that they and their negations can be jointly added to First Order Logic with Team Semantics without increasing its expressive power. It was also shown that these dependencies are the same as the ones that are safe for First Order Logic plus the contradictory negation operator restricted to first order literals and dependency atoms; that they are exactly the ones for which both they and their negations can be \emph{separately} added to First Order Logic with Team Semantics without increasing its expressive power; and that a family of dependencies is doubly strongly first order if and only if every dependency in it is individually doubly strongly first order. 

There are several directions in which this work could be extended. One, for example, could consist in studying more in depth the notion of relativizability, proving that all strongly first order / doubly strongly first order dependencies are relativizable or finding a counterexample, in order to rid the main result of this work (as well as the characterizations of the downwards closed dependencies and non-jumping dependencies that are strongly first order) from the requirement of relativizability or prove its necessity. 

Another direction would be to consider the problem of characterizing safe dependencies in other extensions of First Order Logic based on Team Semantics. In order to further pursue this, it would be useful to develop a classification of Team Semantics-based connectives to mirror and extend the classification of dependencies in various categories (downwards closed, upwards closed, union closed, \ldots) that has proven fruitful in the study of dependencies. It is important to remark, in this context, that even though much of the work in this direction so far has taken for granted the presence of the ``standard'' connectives of First Order Logic (with the Team Semantics rules arising for them through the Game-Theoretic Semantics) there is no pressing reason why that should be the case: as the case of \textbf{FOT} logic \cite{kontinen2019logics} shows, logics that eschew them may also be well worth investigating. 

The problem of characterizing strongly first order dependencies remains a major open questions in this area of research, of course, not only because of its own intrinsic interest but also because it constitutes a useful testbed of our understanding of Team Semantics-based extensions of First Order Logic. 

Finally, quantitative variants of Team Semantics such as the ones investigated in \cite{durand2018approximation,durand2018probabilistic,hannula2019facets,hannula2020descriptive} have been recently attracted much interest, both because of their potential applications and because of their connections to problems in metafinite model theory. The same question about the classification of the logics based on them that have been discussed in the case of ``classical'' Team Semantics could be asked for these semantics, and their answer could do much to illuminate the properties of these more semantics and clarify the advantages and drawbacks of logics built on them.  
%
%
%
 \bibliographystyle{splncs04}
 \bibliography{biblio}

\begin{thebibliography}{10}
\providecommand{\url}[1]{\texttt{#1}}
\providecommand{\urlprefix}{URL }
\providecommand{\doi}[1]{https://doi.org/#1}

\bibitem{abramsky09}
Abramsky, S., V{\"a}{\"a}n{\"a}nen, J.: From {IF} to {BI}. Synthese
  \textbf{167}(2),  207--230 (Mar 2009). \doi{10.1007/s11229-008-9415-6}

\bibitem{armstrong74}
Armstrong, W.W.: {Dependency Structures of Data Base Relationships.} In: Proc.
  of IFIP World Computer Congress. pp. 580--583 (1974)

\bibitem{durand2018approximation}
Durand, A., Hannula, M., Kontinen, J., Meier, A., Virtema, J.: Approximation
  and dependence via multiteam semantics. Annals of Mathematics and Artificial
  Intelligence  \textbf{83}(3),  297--320 (2018)

\bibitem{durand2018probabilistic}
Durand, A., Hannula, M., Kontinen, J., Meier, A., Virtema, J.: Probabilistic
  team semantics. In: International Symposium on Foundations of Information and
  Knowledge Systems. pp. 186--206. Springer (2018)

\bibitem{engstrom12}
Engstr{\"o}m, F.: Generalized quantifiers in dependence logic. Journal of
  Logic, Language and Information  \textbf{21}(3),  299--324 (2012).
  \doi{10.1007/s10849-012-9162-4}

\bibitem{galliani12}
Galliani, P.: Inclusion and exclusion dependencies in team semantics: On some
  logics of imperfect information. Annals of Pure and Applied Logic
  \textbf{163}(1),  68 -- 84 (2012). \doi{10.1016/j.apal.2011.08.005}

\bibitem{galliani2015upwards}
Galliani, P.: Upwards closed dependencies in team semantics. Information and
  Computation  \textbf{245},  124--135 (2015). \doi{10.1016/j.ic.2015.06.008}

\bibitem{galliani2016strongly}
Galliani, P.: On strongly first-order dependencies. In: Dependence Logic, pp.
  53--71. Springer (2016). \doi{10.1007/978-3-319-31803-5\_4}

\bibitem{galliani2019characterizing}
Galliani, P.: Characterizing downwards closed, strongly first-order,
  relativizable dependencies. The Journal of Symbolic Logic  \textbf{84}(3),
  1136--1167 (2019). \doi{10.1017/jsl.2019.12}

\bibitem{galliani2019nonjumping}
Galliani, P.: Characterizing strongly first order dependencies: The non-jumping
  relativizable case. Electronic Proceedings in Theoretical Computer Science
  \textbf{305},  66–82 (Sep 2019). \doi{10.4204/eptcs.305.5},
  \url{http://dx.doi.org/10.4204/EPTCS.305.5}

\bibitem{galliani2020safe}
Galliani, P.: Safe dependency atoms and possibility operators in team
  semantics. Information and Computation p. 104593 (2020)

\bibitem{sep-logic-dependence}
Galliani, P.: {Dependence Logic}. In: Zalta, E.N. (ed.) The {Stanford}
  Encyclopedia of Philosophy. Metaphysics Research Lab, Stanford University,
  summer 2021 edn. (2021)

\bibitem{galliani13b}
Galliani, P., Hannula, M., Kontinen, J.: {Hierarchies in independence logic}.
  In: Rocca, S.R.D. (ed.) Computer Science Logic 2013 (CSL 2013). Leibniz
  International Proceedings in Informatics (LIPIcs), vol.~23, pp. 263--280.
  Schloss Dagstuhl--Leibniz-Zentrum fuer Informatik, Dagstuhl, Germany (2013).
  \doi{10.4230/LIPIcs.CSL.2013.263}

\bibitem{gallhella13}
Galliani, P., Hella, L.: {Inclusion Logic and Fixed Point Logic}. In: Rocca,
  S.R.D. (ed.) Computer Science Logic 2013 (CSL 2013). Leibniz International
  Proceedings in Informatics (LIPIcs), vol.~23, pp. 281--295. Schloss
  Dagstuhl--Leibniz-Zentrum fuer Informatik, Dagstuhl, Germany (2013).
  \doi{10.4230/LIPIcs.CSL.2013.281}

\bibitem{gradel13}
Gr\"adel, E., V\"a\"an\"anen, J.: Dependence and independence. Studia Logica
  \textbf{101}(2),  399--410 (2013). \doi{10.1007/s11225-013-9479-2}

\bibitem{hannula2015hierarchies}
Hannula, M.: Hierarchies in inclusion logic with lax semantics. In: Indian
  Conference on Logic and Its Applications. pp. 100--118. Springer (2015).
  \doi{10.1007/978-3-662-45824-2\_7}

\bibitem{hannula2019facets}
Hannula, M., Hirvonen, {\AA}., Kontinen, J., Kulikov, V., Virtema, J.: Facets
  of distribution identities in probabilistic team semantics. In: European
  Conference on Logics in Artificial Intelligence. pp. 304--320. Springer
  (2019)

\bibitem{hannula2020descriptive}
Hannula, M., Kontinen, J., Van~den Bussche, J., Virtema, J.: Descriptive
  complexity of real computation and probabilistic independence logic. In:
  Proceedings of the 35th Annual ACM/IEEE Symposium on Logic in Computer
  Science. pp. 550--563 (2020)

\bibitem{hintikka96}
Hintikka, J.: The Principles of Mathematics Revisited. Cambridge University
  Press (1996). \doi{10.1017/cbo9780511624919}

\bibitem{hintikkasandu89}
Hintikka, J., Sandu, G.: {I}nformational independence as a semantic phenomenon.
  In: Fenstad, J., Frolov, I., Hilpinen, R. (eds.) Logic, methodology and
  philosophy of science, pp. 571--589. Elsevier (1989).
  \doi{10.1016/S0049-237X(08)70066-1}

\bibitem{hodges97}
Hodges, W.: {C}ompositional {S}emantics for a {L}anguage of {I}mperfect
  {I}nformation. Journal of the Interest Group in Pure and Applied Logics
  \textbf{5 (4)},  539--563 (1997). \doi{10.1093/jigpal/5.4.539}

\bibitem{kontinen2016decidable}
Kontinen, J., Kuusisto, A., Virtema, J.: {Decidability of Predicate Logics with
  Team Semantics}. In: 41st International Symposium on Mathematical Foundations
  of Computer Science (MFCS 2016). Leibniz International Proceedings in
  Informatics (LIPIcs), vol.~58, pp. 60:1--60:14 (2016).
  \doi{10.4230/LIPIcs.MFCS.2016.60}

\bibitem{kontinennu11}
Kontinen, J., Nurmi, V.: Team logic and second-order logic. Fundamenta
  Informaticae  \textbf{106}(2-4),  259--272 (2011)

\bibitem{kontinen2019logics}
Kontinen, J., Yang, F.: Logics for first-order team properties. In:
  International Workshop on Logic, Language, Information, and Computation. pp.
  392--414. Springer (2019)

\bibitem{kuusisto2015double}
Kuusisto, A.: A double team semantics for generalized quantifiers. Journal of
  Logic, Language and Information  \textbf{24}(2),  149--191 (2015)

\bibitem{luck2018complexity}
L{\"u}ck, M.: {On the Complexity of Team Logic and Its Two-Variable Fragment}.
  In: Potapov, I., Spirakis, P., Worrell, J. (eds.) 43rd International
  Symposium on Mathematical Foundations of Computer Science (MFCS 2018).
  Leibniz International Proceedings in Informatics (LIPIcs), vol.~117, pp.
  27:1--27:22. Schloss Dagstuhl--Leibniz-Zentrum fuer Informatik, Dagstuhl,
  Germany (2018). \doi{10.4230/LIPIcs.MFCS.2018.27}

\bibitem{mann11}
Mann, A.L., Sandu, G., Sevenster, M.: Independence-Friendly Logic: A
  Game-Theoretic Approach. Cambridge University Press (2011).
  \doi{10.1017/CBO9780511981418}

\bibitem{ronnholm2018capturing}
R\"onnholm, R.: Capturing k-ary existential second order logic with k-ary
  inclusion-exclusion logic. Annals of Pure and Applied Logic  \textbf{169}(3),
   177 -- 215 (2018). \doi{10.1016/j.apal.2017.10.005}

\bibitem{ronnholm2019expressive}
R{\"o}nnholm, R.: The expressive power of k-ary exclusion logic. Annals of Pure
  and Applied Logic  \textbf{170}(9),  1070--1099 (2019)

\bibitem{vaananen07}
V\"a\"an\"anen, J.: Dependence Logic. Cambridge University Press (2007).
  \doi{10.1017/CBO9780511611193}

\bibitem{vaananen07b}
V\"a\"an\"anen, J.: {T}eam {L}ogic. In: van Benthem, J., Gabbay, D., L\"owe, B.
  (eds.) {I}nteractive {L}ogic. {S}elected {P}apers from the 7th {A}ugustus de
  {M}organ {W}orkshop, pp. 281--302. {A}msterdam {U}niversity {P}ress (2007)

\bibitem{yang10}
Yang, F.: Expressing second-order sentences in intuitionistic dependence logic.
  In: Kontinen, J., V\"a\"an\"anen, J. (eds.) Proceedings of {D}ependence and
  {I}ndependence in {L}ogic, pp. 118--132. ESSLLI 2010 (2010)

\end{thebibliography}
\end{document}